\documentclass[12pt]{amsart}
\usepackage{latexsym, amsmath, amssymb, amsthm, amsfonts, amstext, amscd, amsxtra, amsbsy, mathrsfs, stmaryrd, mathdots, bbold}
\usepackage[all]{xy}
\usepackage{tikz}
\usepackage{soul}
\usepackage{bbm}
\usepackage{color}
\usepackage{hyperref}
\usepackage{mathtools}
\usepackage{yhmath}

\allowdisplaybreaks

\newtheorem{theorem}{Theorem}[section]
\newtheorem{lemma}[theorem]{Lemma}
\newtheorem{prop}[theorem]{Proposition}

\theoremstyle{definition}

\newtheorem{remark}[theorem]{Remark}

\newcommand{\nc}{\newcommand}
\nc{\rnc}{\renewcommand}
\nc{\bb}[1]{{\mathbb #1}}
\nc{\bbA}{\bb{A}}\nc{\bbB}{\bb{B}}\nc{\bbC}{\bb{C}}\nc{\bbD}{\bb{D}}
\nc{\bbE}{\bb{E}}\nc{\bbF}{\bb{F}}\nc{\bbG}{\bb{G}}\nc{\bbH}{\bb{H}}
\nc{\bbI}{\bb{I}}\nc{\bbJ}{\bb{J}}\nc{\bbK}{\bb{K}}\nc{\bbL}{\bb{L}}
\nc{\bbM}{\bb{M}}\nc{\bbN}{\bb{N}}\nc{\bbO}{\bb{O}}\nc{\bbP}{\bb{P}}
\nc{\bbQ}{\bb{Q}}\nc{\bbR}{\bb{R}}\nc{\bbS}{\bb{S}}\nc{\bbT}{\bb{T}}
\nc{\bbU}{\bb{U}}\nc{\bbV}{\bb{V}}\nc{\bbW}{\bb{W}}\nc{\bbX}{\bb{X}}
\nc{\bbY}{\bb{Y}}\nc{\bbZ}{\bb{Z}}
\nc{\mbf}[1]{{\mathbf #1}}
\nc{\bfA}{\mbf{A}}\nc{\bfB}{\mbf{B}}\nc{\bfC}{\mbf{C}}\nc{\bfD}{\mbf{D}}
\nc{\bfE}{\mbf{E}}\nc{\bfF}{\mbf{F}}\nc{\bfG}{\mbf{G}}\nc{\bfH}{\mbf{H}}
\nc{\bfI}{\mbf{I}}\nc{\bfJ}{\mbf{J}}\nc{\bfK}{\mbf{K}}\nc{\bfL}{\mbf{L}}
\nc{\bfM}{\mbf{M}}\nc{\bfN}{\mbf{N}}\nc{\bfO}{\mbf{O}}\nc{\bfP}{\mbf{P}}
\nc{\bfQ}{\mbf{Q}}\nc{\bfR}{\mbf{R}}\nc{\bfS}{\mbf{S}}\nc{\bfT}{\mbf{T}}
\nc{\bfU}{\mbf{U}}\nc{\bfV}{\mbf{V}}\nc{\bfW}{\mbf{W}}\nc{\bfX}{\mbf{X}}
\nc{\bfY}{\mbf{Y}}\nc{\bfZ}{\mbf{Z}}
\nc{\bfa}{\mbf{a}}\nc{\bfb}{\mbf{b}}\nc{\bfc}{\mbf{c}}\nc{\bfd}{\mbf{d}}
\nc{\bfe}{\mbf{e}}\nc{\bff}{\mbf{f}}\nc{\bfg}{\mbf{g}}\nc{\bfh}{\mbf{h}}
\nc{\bfi}{\mbf{i}}\nc{\bfj}{\mbf{j}}\nc{\bfk}{\mbf{k}}\nc{\bfl}{\mbf{l}}
\nc{\bfm}{\mbf{m}}\nc{\bfn}{\mbf{n}}\nc{\bfo}{\mbf{o}}\nc{\bfp}{\mbf{p}}
\nc{\bfq}{\mbf{q}}\nc{\bfr}{\mbf{r}}\nc{\bfs}{\mbf{s}}\nc{\bft}{\mbf{t}}
\nc{\bfu}{\mbf{u}}\nc{\bfv}{\mbf{v}}\nc{\bfw}{\mbf{w}}\nc{\bfx}{\mbf{x}}
\nc{\bfy}{\mbf{y}}\nc{\bfz}{\mbf{z}}

\nc{\mcal}[1]{{\mathcal #1}}
\nc{\calA}{\mcal{A}}\nc{\calB}{\mcal{B}}\nc{\calC}{\mcal{C}}\nc{\calD}{\mcal{D}}
\nc{\calE}{\mcal{E}} \nc{\calF}{\mcal{F}}\nc{\calG}{\mcal{G}}\nc{\calH}{\mcal{H}}
\nc{\calI}{\mcal{I}}\nc{\calJ}{\mcal{J}}\nc{\calK}{\mcal{K}}\nc{\calL}{\mcal{L}}
\nc{\calM}{\mcal{M}}\nc{\calN}{\mcal{N}}\nc{\calO}{\mcal{O}}\nc{\calP}{\mcal{P}}
\nc{\calQ}{\mcal{Q}}\nc{\calR}{\mcal{R}}\nc{\calS}{\mcal{S}}\nc{\calT}{\mcal{T}}
\nc{\calU}{\mcal{U}}\nc{\calV}{\mcal{V}}\nc{\calW}{\mcal{W}}\nc{\calX}{\mcal{X}}
\nc{\calY}{\mcal{Y}}\nc{\calZ}{Z}
\nc{\fA}{\frak{A}}\nc{\fB}{\frak{B}}\nc{\fC}{\frak{C}} \nc{\fD}{\frak{D}}
\nc{\fE}{\frak{E}}\nc{\fF}{\frak{F}}\nc{\fG}{\frak{G}}\nc{\fH}{\frak{H}}
\nc{\fI}{\frak{I}}\nc{\fJ}{\frak{J}}\nc{\fK}{\frak{K}}\nc{\fL}{\frak{L}}
\nc{\fM}{\frak{M}}\nc{\fN}{\frak{N}}\nc{\fO}{\frak{O}}\nc{\fP}{\frak{P}}
\nc{\fQ}{\frak{Q}}\nc{\fR}{\frak{R}}\nc{\fS}{\frak{S}}\nc{\fT}{\frak{T}}
\nc{\fU}{\frak{U}}\nc{\fV}{\frak{V}}\nc{\fW}{\frak{W}}\nc{\fX}{\frak{X}}
\nc{\fY}{\frak{Y}}\nc{\fZ}{\frak{Z}}
\nc{\fa}{\frak{a}}\nc{\fb}{\frak{b}}\nc{\fc}{\frak{c}} \nc{\fd}{\frak{d}}
\nc{\fe}{\frak{e}}\nc{\fFf}{\frak{f}}\nc{\fg}{\frak{g}}\nc{\fh}{\frak{h}}
\nc{\fri}{\frak{i}}\nc{\fj}{\frak{j}}\nc{\fk}{\frak{k}}\nc{\fl}{\frak{l}}
\nc{\fm}{\frak{m}}\nc{\fn}{\frak{n}}\nc{\fo}{\frak{o}}\nc{\fp}{\frak{p}}
\nc{\fq}{\frak{q}}\nc{\fr}{\frak{r}}\nc{\fs}{\frak{s}}\nc{\ft}{\frak{t}}
\nc{\fu}{\frak{u}}\nc{\fv}{\frak{v}}\nc{\fw}{\frak{w}}\nc{\fx}{\frak{x}}
\nc{\fy}{\frak{y}}\nc{\fz}{\frak{z}}

\newcommand{\mbb}{\mathbb}
\newcommand{\mscr}{\mathscr}
\newcommand{\mrm}{\mathrm}

\newcommand{\Gr}{\mrm Gr}
\newcommand{\Sym}{\operatorname{Sym}\nolimits}
\newcommand{\Res}{\operatorname{Res}\nolimits}

\newcommand{\End}{\mrm{End}}

\def \C{{\mathbb C}}

\newcommand{\qbinom}[2]{\begin{bmatrix} #1\\#2 \end{bmatrix} }

\newcommand{\K}{\mathbb K}

\newcommand{\arxiv}[1]{\href{http://arxiv.org/abs/#1}{\tt arXiv:\nolinkurl{#1}}}
\newcommand{\Z}{\mathbb Z}

\newcommand{\bB}{{\mathbf B }} 
\newcommand{\bDel}{\boldsymbol{\Delta}}
\newcommand{\bTh}{\boldsymbol{\Theta}}

\newcommand{\tUi}{\widetilde{{\mathbf U}}^\imath}

\def\ro{\text{ro}}
\def\co{\text{co}}
\def\mf{\mathfrak}
\def\diag{\text{diag}}
\def \fg{\mathfrak{g}}

\DeclareMathOperator{\Span}{Span}
\DeclareMathOperator{\IC}{IC}
\DeclareMathOperator{\pt}{pt}
\DeclareMathOperator{\Sp}{Sp}

\title[Affine $\mathrm{i}$quantum groups and Steinberg varieties of type C II]{Affine $\mathrm{i}$quantum groups and Steinberg varieties \\ of type C, II}

\author{Li Luo}
 \address[Li Luo]{School of Mathematical Sciences, Key Laboratory of MEA (Ministry of Education) \& Shanghai Key Laboratory of PMMP, East China Normal University, Shanghai 200241, China}
 \email{lluo@math.ecnu.edu.cn}
 
\author{Changjian Su}
 \address[Changjian Su]{Yau Mathematical Sciences Center, Tsinghua University, Beijing, 100084, China}\email{changjiansu@mail.tsinghua.edu.cn}

\author{Zheming Xu}
 \address[Zheming Xu]{School of Mathematical Sciences, Key Laboratory of MEA (Ministry of Education) \& Shanghai Key Laboratory of PMMP, East China Normal University, Shanghai 200241, China}
 \email{zxu0@stu.ecnu.edu.cn}
 
\date{\today}

\begin{document}

\begin{abstract}
A geometric realization of the quasi-split affine iquantum group of type $\mathrm{AIII}_{2n-1}^{(\tau)}$ was given by Wang and the second author, in terms of equivariant K-groups of Steinberg varieties of type C. As a completion of that work, this paper focuses on the previously untreated case. We provide a similar construction of the quasi-split affine iquantum group of type $\mathrm{AIII}_{2n}^{(\tau)}$, using the same equivariant K-groups of Steinberg varieties of type C. 
In the appendix, we employ Steinberg varieties of type D to give a new realization of the quasi-split affine iquantum group of type $\mathrm{AIII}_{2n-1}^{(\tau)}$, thereby avoiding the localization method adopted in the previous work.
\end{abstract}
\maketitle


\section{Introduction}
\subsection{History}
In geometric representation theory, realizing algebraic structures via geometric objects often serves as a first step in an investigation. Iwahori's construction of Hecke algebras (see \cite{Iw64}) via $G$-orbits on double complete flag varieties stands as an early classic in this area. In the context of Langlands reciprocity, another geometric realization of Hecke algebras arises from equivariant K-groups of Steinberg varieties by Kazhdan and Lusztig \cite{Lu85, KL85, KL87} (see also \cite{CG97}).

In type A, the above two geometric approaches to Hecke algebras have been extended to the quantum group $\mathbf{U}(\mathfrak{gl}_N)$ and its affinization. Specifically, using the $N$-step flag variety in place of the complete flag variety, Beilinson, Lusztig and MacPherson (abbr. BLM) \cite{BLM90} constructed the quantum Schur algebra (of type A) and furthermore $\mathbf{U}(\mathfrak{gl}_N)$ (see also \cite{Lu99, DF15} for the affine version). Inspired by their work, Ginzburg and Vasserot \cite{GV93,Vas98} realized the affine quantum $\mathfrak{gl}_N$ or $\mathfrak{sl}_N$ via the equivariant K‑theory of the Steinberg variety associated with the $N$-step flag variety. Notably, the BLM-type realization yields the Serre presentation of the affine quantum $\mathfrak{gl}_N$, while the equivariant K-theoretic approach reflects its Drinfeld new presentation \cite{Dr87}.

Nakajima \cite{Na01} provided an equivariant K-theoretic approach for affine quantum groups of types ADE using his quiver varieties instead of flag varieties. However, the aforementioned geometric constructions via flag varieties remained elusive beyond type A for nearly two decades. It was not until the foundational work of Bao and Wang on iquantum groups \cite{BW18} that it became clear that the counterparts of these constructions on flag varieties of types BCD are not the Drinfeld quantum groups, but rather the iquantum groups $\mathbf{U}^\imath_N$ arising from the quantum symmetric pairs $(\mathbf{U}^\imath_N,\mathbf{U}(\mathfrak{gl}_N))$ of quasi-split type AIII (in the sense of Satake diagrams). 

Using flag varieties of type B/C, Bao, Kujawa, Li and Wang \cite{BKLW18} provided a BLM-type realization of $\mathbf{U}^\imath_N$ (see \cite{FL15, DLZ25} for type D). The affine version $\widetilde{\mathbf{U}}^\imath_N$, arising from the quasi-split affine quantum symmetric pair $(\widetilde{\mathbf{U}}^\imath_N, \mathbf{U}(\widetilde{\mathfrak{gl}}_N))$ of type AIII,  was later treated by the first author and his collaborators \cite{FLLLW20,FLLLW23}. 
Fan, Ma and Xiao \cite{FMX22} made an attempt toward the equivariant K-theoretic construction of iquantum groups via the Steinberg variety $Z$ associated with the $N$-step isotropic flag variety for $G=\mathrm{Sp}_{2d}$, but they only handled the finite type case due to the lack of a Drinfeld new presentation of the affine iquantum groups at that time. Instead, Yang and two of the present authors \cite{LXY26} proved that the equivariant K-group $K^{G\times \mathbb{C}^*}(Z)$ is indeed isomorphic to the quantum Schur algebra of affine type C introduced in \cite{CLW24}. In fact, they established such an isomorphism in great generality, thereby providing a Schur algebra analogue of the Langlands reciprocity for an arbitrary Lie type. 

Now that the Drinfeld new presentation of affine iquantum groups becomes available through the works \cite{LWZ23,LWZ24,LPWZ25}, the second author and Wang \cite{SW24} have established an algebra homomorphism from $\widetilde{\mathbf{U}}^\imath_{2n}$ to a localization of $K^{G\times \mathbb{C}^*}(Z)$ with $N=2n$ even. The reason for using localization therein is that a non-closed $G$-orbit on $\mathscr{F}\times \mathscr{F}$ has to be employed to construct algebra generators.
The present paper completes \cite{SW24} by treating the odd case $N=2n+1$ to give an equivariant K-theoretic construction of $\widetilde{\mathbf{U}}^\imath_{2n+1}$. In this case, the main results can be stated without localization since no non-closed $G$-orbit is necessary. Moreover, in the appendix, we find a new path to revisit the even case $N=2n$ by taking $G=O_{2d}$ to avoid the localization method used in \cite{SW24}. 

\subsection{Overview}
Take $G=\mathrm{Sp}_{2d}$ and $T\subset G$ a maximal torus. Let $N=2n+1$ be odd. We consider the $N$-step isotropic flag variety $\mathscr{F}$ of $G$ and the associated Steinberg variety $Z$. Let $\widetilde{\mathbf{U}}^\imath=\widetilde{\mathbf{U}}^\imath_{2n+1}$ be the quasi-split affine iquantum group corresponding to the following affine Satake diagram:
\begin{figure}[htbp]
\centering
\begin{tikzpicture}[scale=.4]

\draw[blue, thick, ->] (-4.5, -1) arc[start angle=30, end angle=330, radius=0.7];

\draw[thick] (-4, -1.5) circle (0.3);
\node at (-4, -2.5) {$0$};
\node at (0,0.75) {$1$};
\node at (4,0.75) {$2$};
\node at (12,0.75) {$n-1$};
\node at (16,0.75) {$n$};     
\node at (0,-3.75) {$2n$};
\node at (4,-3.75) {$2n-1$};
\node at (12,-3.75) {$n+2$};
\node at (16,-3.75) {$n+1$}; 
       
\node at (8,0) {$\dots$};
\node at (8,-3) {$\dots$};

\draw[thick] (-0.28, -0.11) -- (-3.7, -1.22);   
\draw[thick] (-0.28, -2.89) -- (-3.7, -1.78);  
    
\foreach \x in {0,2,6,8} 
{\draw[thick,xshift=\x cm] (\x, 0) circle (0.3); 
\draw[thick,xshift=\x cm] (\x, -3) circle (0.3);}

\foreach \x in {0,6}
{\draw[thick,xshift=\x cm] (\x,0) ++(0.5,0) -- +(3,0);
\draw[thick,xshift=\x cm] (\x,-3) ++(0.5,0) -- +(3,0);}
   
\foreach \x in {2,4.5}
{\draw[thick,xshift=\x cm] (\x,0) ++(0.5,0) -- +(2,0);
\draw[thick,xshift=\x cm] (\x,-3) ++(0.5,0) -- +(2,0);}
    
\foreach \x in {8}
\draw[thick,xshift=\x cm] (\x,-0.5) -- +(0,-2);
    
\foreach \x in {0,4} 
\draw[thick,<->, blue, bend left=50] (\x+0.3,-2.5) to (\x+0.3,-0.5);
\foreach \x in {12,16} 
\draw[thick,<->, blue, bend left=50] (\x-0.3,-2.5) to (\x-0.3,-0.5);

\end{tikzpicture},
\caption{Affine type $\mathrm{AIII}_{2n}^{(\tau)}$}
\label{fig:AIII2n}
\end{figure}
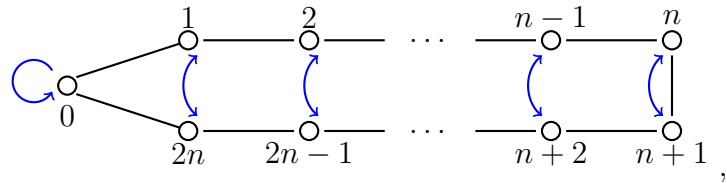
\\
where the blue arrows indicate the diagram involution $\tau$. 

The diagonal $G$-orbits on $\mathscr{F}\times\mathscr{F}$ are indexed by the centrally symmetric $N\times N$ matrices over $\mathbb{N}$ whose entries sum to $2d$. In particular, the orbits $\mathcal{O}_{E_{i,i+1}^\theta(\mathbf{v},a)}$ corresponding to the matrices $E_{i,i+1}^\theta(\mathbf{v},a)$ whose off-diagonal entries are all zero except at the $(i,i+1)$-th and $(N+1-i,N-i)$-th positions will be employed to define some geometric operators. In contrast to \cite{SW24}, these orbits $\mathcal{O}_{E_{i,i+1}^\theta(\mathbf{v},a)}$ are all closed. Thus, we can directly introduce the geometric B-operators (refer to \S\ref{sec:GeoB}) without using localization, which act on $\mathbf{P}\simeq K^{G\times\mathbb{C}^*}(\mathscr{F})$ by the natural convolution product on equivariant K-groups. These B-operators, together with the $\Theta$-operators and $\mathbb{K}$-operators defined in \S\ref{sec:comboper}, determine a representation of $\widetilde{\mathbf{U}}^\imath_{2n+1}$ on $\mathbf{P}\otimes_{\mathbb{C}[q,q^{-1}]}\mathbb{C}(q)$ (see Theorem~\ref{thm:polyrep}), which implies an explicit algebra homomorphism from $\widetilde{\mathbf{U}}^\imath_{2n+1}$ to the equivariant K-group $K^{G\times\mathbb{C}^*}(Z)\otimes_{\mathbb{C}[q,q^{-1}]}\mathbb{C}(q)$ (see Theorem~\ref{thm:main}). 
The specialization of this homomorphism at $(s,t)\in T\times\mathbb{C}^*$ is surjective for $q=t$ not a root of unity (see Theorem~\ref{thm:surj}). Such surjectivity allows us to construct standard and simple modules of $\widetilde{\mathbf{U}}^\imath_{2n+1}|_{q=t}$. The composition multiplicities of these standard modules are given in terms of dimensions of intersection cohomology groups (see Theorem~\ref{thm:simplestandard}).

For the convenience of the readers, we have placed the verification of the involved relations in the Drinfeld new presentation at the end of the main text (see \S\ref{sec:checkrelations}). 
This portion of the calculation exhibits the very technical challenge in the paper. The case of $\widetilde{\mathbf{U}}^\imath_{2n+1}$ is more difficult than that of $\widetilde{\mathbf{U}}^\imath_{2n}$. The Drinfeld new presentation of $\widetilde{\mathbf{U}}^\imath_{2n+1}$ became available recently \cite{LPWZ25}, which contains a new and intricate Serre relation. Verifying this new relation geometrically is a nontrivial challenge. Actually, we no longer work as economically as in \cite{SW24} on the set partitions of $\{1,\ldots,d\}$. Instead, we work on those of the larger interval $\{1,\ldots,2d\}$, even though the information on $\{1,\ldots,d\}$ is essentially equivalent to that on $\{1,\ldots,2d\}$ under the action of the Weyl group of type $C_d$. This treatment seems more cumbersome but in fact makes the explicit computation more straightforward, because it renders certain hidden combinatorial structures for type C less likely to be overlooked.

The main text resolves the first open question posed in \cite[\S1.4]{SW24}. The second issue raised therein is the expectation of using isotropic flag varieties with the orthogonal group action to give an equivariant K-theoretic realization of the affine iquantum groups $\widetilde{\mathbf{U}}^\imath_{N}$. This is precisely what we address in Appendix~\ref{append}. We would like to emphasize that we do not simply carry out a parallel generalization from symplectic groups to orthogonal groups. We deliberately take $G=\mathrm{O}_{2d}$ rather than its connected component $\mathrm{SO}_{2d}$. Thanks to the disconnectedness of $G$, some interesting structure emerges. Particularly, its Weyl group $W=N_G(T)/T$ is of type C instead of type D, and the generator $s_d\in W$, which lies outside the type D Weyl group, corresponds to a closed $G$-orbit on the double complete flag variety. This phenomenon indicates that $G$-orbits $\mathcal{O}_{E_{i,i+1}^\theta(\mathbf{v},a)}$ on $\mathscr{F}\times\mathscr{F}$ under consideration are all closed in this case no matter $N$ is even or odd. Therefore, for even $N=2n$, we can find a new path to obtain the equivariant K-theoretic realization of $\widetilde{\mathbf{U}}^\imath_{2n}$ avoiding the localization.

\subsection{Outline}
Section~\ref{sec:Stein} recalls the convolution algebra structure on the equivariant K-group of the Steinberg variety of type C, as well as some related geometric and combinatorial properties.

In Section~\ref{sec:operators}, we introduce some operators and describe their actions on $K^{G\times\mathbb{C}^*}(\mathscr{F})$ explicitly. These operators are the images of the generators of $\widetilde{\mathbf{U}}^\imath_{2n+1}$ under the equivariant K-theoretic realization.

Section~\ref{sec:main} is devoted to the main results. We give an algebra homomorphism from $\widetilde{\mathbf{U}}^\imath_{2n+1}$ to the equivariant K-group of the Steinberg variety. As an application, we construct a family of finite-dimensional standard modules of $\widetilde{\mathbf{U}}^\imath_{2n+1}$ 
and give their composition multiplicities.

In Appendix~\ref{append}, we study a type D framework, by which we find a new path to obtain the equivariant K-theoretic realization of $\widetilde{\mathbf{U}}^\imath_{2n}$ that avoids the localization method.

\subsection*{Acknowledgments}
This work is partially supported by the National Key R\&D Program of China (No. 2025YFA1017400) 
and the NSF of China (No. 12371028).

\section{Convolution algebra of the Steinberg variety}
\label{sec:Stein}

In this section, we review the basics on convolution constructions in equivariant K-theory and apply them to the Steinberg variety $\calZ$ of type $C$ (cf. \cite{CG97, Vas98}).

\subsection{Convolution in equivariant K-theory}
\label{subsec:convolution}

For a connected complex reductive algebraic group $G$ and a quasi-projective $G$-variety $X$, let $K^G(X)$ denote the complexified $G$-equivariant K-group of $X$, see \cite{CG97}. 
If $X= \{\pt\}$ is a point, $K^G(\pt) = R(G)$, the complexified representation ring of $G$.

Given three smooth $G$-varieties $M_1,~M_2,~M_3$, let
$$p_{ij}:M_1 \times M_2 \times M_3 \rightarrow M_i\times M_j$$
be the obvious projection maps.
Let $Z_{12}\subseteq M_1 \times M_2 $ and $Z_{23}\subseteq M_2 \times M_3 $ be $G$-stable closed subvarieties.
We denote
 \begin{equation*}
   Z_{12}\circ Z_{23}= p_{13}(p_{12}^{-1}(Z_{12})\cap p_{23}^{-1}(Z_{23})).
 \end{equation*}
If the restriction of $p_{13}$ to $p_{12}^{-1}(Z_{12})\cap p_{23}^{-1}(Z_{23})$ is a proper map,
then we define the convolution product as follows:
\begin{equation*}
\begin{split}
\star: \ \ K^{G}(Z_{12})\otimes K^{G}(Z_{23})&\longrightarrow K^{G}(Z_{12} \circ Z_{23}), \\
   \mscr{F}_{1} \otimes \mscr{F}_{2} &\mapsto p_{{13}{\ast}}(p_{12}^{\ast}\mscr{F}_{1} \otimes p_{23}^{\ast}\mscr{F}_{2}),
\end{split}
\end{equation*}
where all the functors here and below are understood to be derived.

Let $F_i$ ($i=1,2$) be smooth $G$-varieties, $M_i=T^*F_i$, and $\pi_i$ denote the projections $M_i\rightarrow F_i$. The torus $\bbC^*$ acts on $M_i$ by $z\cdot (x,\xi)=(x,z^{-2}\xi)$, where $x\in F_i$ and $\xi\in T_x^*F_i$. By definition, $K^{\bbC^*}(\pt)=\bbC[q,q^{-1}]$, where $q$ corresponds to the standard representation of $\bbC^*$. Let $\calO\subset F_1\times F_2$ be a smooth $G$-variety, and $Z_\calO$ denote the conormal bundle $T_\calO^*(F_1\times F_2)\subset M_1\times M_2$. Suppose the projection $Z_\calO\rightarrow M_1$ is proper and the projections $p_{i,\calO}:\calO\rightarrow F_i$ are smooth fibrations with $p_{1,\calO}$ being proper. By the Thom isomorphism,  $K^{G\times\bbC^*}(Z_\calO)\simeq K^{G\times\bbC^*}(\calO)$. Therefore, any $\mscr{K}\in K^{G\times\bbC^*}(\calO)$ defines an $R(G\times\bbC^*)$-modules homomorphism $\rho_{\mscr K}:K^{G\times\bbC^*}(M_2)\rightarrow K^{G\times \bbC^*}(M_1)$ by convolution. We have the following useful formula. 

\begin{lemma}\cite[Corollary 4]{Vas98}\label{lem:convolution}
For any $\mscr{K}\in K^{G\times\bbC^*}(\calO)$ and $\mscr F\in K^{G\times \bbC^*}(F_2)$,
\[\rho_\mscr K(\pi_2^*\mscr F)=\pi_1^* p_{1,\calO*}\bigg(\bigwedge\nolimits_{q^2}T_{p_{1,\calO}}\otimes p_{2,\calO}^*\mscr F\otimes \mscr K\bigg),\]
where $T_{p_{1,\calO}}$ is the relative tangent sheaf along the fibers of $p_{1,\calO}$, and $\bigwedge_{q^2}T_{p_{1,\calO}}=\sum_i (-q^2)^i \bigwedge^i T_{p_{1,\calO}}$. 
\end{lemma}

For the computations, we will use frequently the localization formula in equivariant K-theory. Let $T\subset G$ be a maximal torus, and let $X$ be a smooth projective variety such that the torus fixed point set $X^T$ is finite. First of all, we have $K^G(X)\simeq K^T(X)^W$, where $W$ is the Weyl group. Let $\pi:X\rightarrow \pt$ be the structure morphism. Then for any $\mscr F\in K^T(X)$, we have the following localization formula \cite{CG97}
\begin{equation*}
    \pi_*(\mscr F)=\sum_{x\in X^T}\frac{\mscr F|_x}{\bigwedge^\bullet (T_x^*X)}\in K^T(\pt),
\end{equation*}
where $\mscr F|_x\in K^T(\pt)$ is the pullback of $\mscr F$ to the fixed point $x\in X^T$, and $\bigwedge^\bullet T_x^*X=\sum_i(-1)^i\wedge^i(T^*_xX)=\prod_{\mu_i}(1-e^{\mu_i})\in K^T(\pt)$ with the product over all the torus weights $\{\mu_i\}$ in the $T$-vector space $T_x^*X$.

\subsection{Partial flag varieties of type C}
\label{subsec:flag}

Let $V :=\C^{2d}$ equip with a non-degenerate skew-symmetric bilinear form $(-,-)$ given by the matrix 
$\begin{pmatrix}
    0& J_d\\-J_d&0
\end{pmatrix},$
where \[J_d=\begin{pmatrix}
    & & 1\\
    & \adots &\\
    1 & &
\end{pmatrix}_{d\times d}.\]
Let $\{\epsilon_i\mid 1\leq i\leq 2d\}$ denote the standard basis of $V$.

Throughout the paper, we set
\[
G = \Sp(V) \quad\text{ and } \quad N=2n+1,
\]
for a fixed positive integer $n$. Let 
\begin{align*}  
\Lambda^\imath_{N,d}= \big\{ \mathbf{v}=(v_i) \in \mathbb{N}^{N} \mid   v_i = v_{N+1-i},\quad \textstyle \sum_{i=1}^{N} v_i = 2d \big\}.
\end{align*}
For any subspace $W\subseteq V$, let $W^{\perp} = \{x \in V\mid (x,y)=0, \ \forall y\in W\}$.
For any $\mathbf{v} \in \Lambda^\imath_{N,d}$, define 
$$
\mathscr{F}_{\mathbf{v}}=\{ F=( 0=V_0 \subset V_1 \subset \cdots \subset V_{N}=V)\ \mid \  V_i=V_{N-i}^{\perp} ,\ \text{dim}( V_{i}/V_{i-1})= v_i,\  \forall i \}. 
$$
The natural $G$-action on $V$ induces a natural transitive action of $G$ on $\mathscr{F}_{\mathbf{v}}$, and thus 
\begin{align*}
\mathscr{F} = \bigsqcup _{\mathbf{v} \in \Lambda^\imath_{N,d}} \mathscr{F}_{\mathbf{v}} 
\end{align*}
is a $G$-variety called the {\em $N$-step partial flag variety}. Let $F_\mathbf{v}$ be the flag in $\mathscr{F}_{\mathbf{v}}$ such that for $1\leq i\leq n$, $V_i=\Span\{\epsilon_j\mid j\in [\mathbf{v}]_i\}$ and $V_{N-i}=V_i^\perp$. Let $P_\mathbf{v}$ be the stabilizer of the flag $F_\mathbf{v}$ inside $G$, then 
\begin{equation*}
    G/P_{\mathbf{v}}\simeq \mathscr{F}_{\mathbf{v}}.
\end{equation*}

For any $\mathbf{v}=(v_1,\cdots,v_N)\in \Lambda^\imath_{N,d}$, denote $\bar{v}_i:=v_1+\cdots v_i$, $[\mathbf{v}]_i=[1+\bar{v}_{i-1},\bar{v}_i]$ for $1\leq i\leq N$, and $[\mathbf{v}]^{\mathfrak{c}}_{n+1}:=[1+\bar{v}_n,d]$.
Let $$[\mathbf{v}]:=([\mathbf{v}]_1, [\mathbf{v}]_2, \ldots,[\mathbf{v}]_N),$$ 
which is a partition of the set $\{1,2,\cdots,2d\}$.

Let $W_{\mathfrak{c}} = \mathbb{Z}_2^{d}\rtimes \mathfrak{S}_{d}$ be the Weyl group of type $C_d$, which
has a natural action on the set $\{1,2, \ldots,2d\}$.
Denote $$W_{[\mathbf{v}]}:=\mathfrak{S}_{[\mathbf{v}]_1}\times\cdots\times \mathfrak{S}_{[\mathbf{v}]_n}\times(\mathbb{Z}_{2}^{|[\mathbf{v}]^{\mathfrak{c}}_{n+1}|}\rtimes \mathfrak{S}_{[\mathbf{v}]_{n+1}^{\mathfrak{c}}})\subset W_{\mathfrak{c}},$$ which is the parabolic subgroup of $W_\mathfrak{c}$ corresponding to the parabolic subgroup $P_\mathbf{v}$. 

\subsection{Diagonal $G$-orbits on $\mathscr{F}\times\mathscr{F}$}
Denote
\begin{align*}
\Xi_{N,d}=\Big\{A=(a_{ij}) \in {\rm Mat}_{N\times N}(\mathbb{N})\ | \ \sum_{i,j}a_{ij}=2d,\ a_{ij}=a_{N+1-i,N+1-j},\ \forall \  i, j \Big\}.
\end{align*}
To each matrix $A \in \Xi_{N,d}$, we associate a partition of the set $\{1,2,\cdots,2d\}$ as follows
    \begin{equation*}
      [A]=([A]_{11}, \cdots [A]_{1N}, [A]_{21},\cdots, [A]_{NN}),
    \end{equation*}
where $[A]_{ij} =[\sum\limits_{(h,k)< (i,j)}a_{hk}+1 , \sum\limits_{(h,k)<(i,j)}a_{hk}+a_{ij}]\subseteq \mathbb{N}$,
and $<$ is the  left lexicographical order, i.e.,
$$(h,k) < (i,j) \Leftrightarrow h<i  ~\text{or}~ (h = i ~\text{and}~  k<j).$$ Moreover, let $[A]^{\mathfrak{c}}_{n+1,n+1}= [\sum_{(h,k)< (n+1,n+1)}a_{hk}+1, d]$.

Define a parabolic subgroup of $W_\mathfrak{c}$ for each $A\in\Xi_{N,d}$ by 
\[W_{[A]}:=\mathfrak{S}_{[A]_{11}}\times\dotsb \times \mathfrak{S}_{[A]_{1,N}}\times \mathfrak{S}_{[A]_{21}} \times\dotsb\times \mathfrak{S}_{[A]_{n+1,n}}\times(\mathbb{Z}_2^{|[A]_{n+1,n+1}^{\mathfrak{c}}|}\rtimes \mathfrak{S}_{[A]_{n+1,n+1}^{\mathfrak{c}}}).\]

For any matrix $A \in \Xi_{N,d}$, denote
\begin{equation*}
  \ro(A) = (\sum_j a_{ij})_{i=1, 2,\cdots, N}\in\Lambda_{N,d}^\imath,\quad\ {\rm and}\  \co(A) = (\sum_i a_{ij})_{j=1,2, \cdots, N}\in\Lambda_{N,d}^\imath.
\end{equation*}
For any $\mathbf{v}, \mathbf{w} \in\Lambda_{N,d}^\imath$, let
$$\Xi_{N,d}(\mathbf{v},\mathbf{w}) = \{ A \in \Xi_{N,d} \mid \ro(A) = \mathbf{v},\  \co(A) = \mathbf{w} \}.$$

For a pair of flags $(F,F')\in \mathscr{F}_{\mathbf{v}}\times \mathscr{F}_{\mathbf{w}}$, define an $N\times N$ matrix $A=(a_{i,j})$ by setting
\begin{equation*}
    a_{i,j}=\dim \frac{V_i\cap V_j'}{V_{i-1}\cap V_j'+V_{i}\cap V_{j-1}'}.
\end{equation*}
It has been shown in \cite[Section 6]{BKLW18} that this gives a bijection between the diagonal $G$-orbits in $\mathscr{F}_{\mathbf{v}}\times \mathscr{F}_{\mathbf{w}}$ and $\Xi_{N,d}(\mathbf{v},\mathbf{w})$. For any $A\in \Xi_{N,d}$, let $\mathscr{O}_A$ denote the corresponding $G$-orbit on $\mathscr{F}\times \mathscr{F}$.
On the other hand, it is well known that the diagonal $G$-orbits on $\mathscr{F}_{\mathbf{v}}\times \mathscr{F}_{\mathbf{w}}$ are in bijection with the double cosets $W_{[\mathbf{v}]}\backslash W/W_{[\mathbf{w}]}$. 
 
We can define an order $\preceq$ on $\Xi_{N,d}$ as follows.
For any $A=(a_{ij}), B=(b_{ij})\in \Xi_{N,d}$, $A \preceq B $ if and only if
\begin{align}\label{equ:order}
\ro(A) = \ro(B),\ \co(A)=\co(B),\ {\rm and}\ \sum_{r\leq i; s\geq j} a_{rs} \leq \sum_{r\leq i; s\geq j} b_{rs}, \  \forall  i<j.
\end{align}
This order is compatible with the Bruhat order on $W_{\mathfrak{c}}$ via the above bijection.
\begin{prop}\cite{BKLW18}\label{prop:orbit}
For any $A,B \in \Xi_{N,d}$, $A \preceq B$ if $\mathcal{O}_A \subseteq  \overline{\mathcal{O}}_B$.
\end{prop}

Let $E_{ij}$ be the standard $N\times N$ matrix unit with 1 at $(i,j)$-entry. For $\mathbf{v}\in \Lambda_{\mathfrak{c},d-a}$,  define 
\begin{align} \label{eq:Eijv}
E_{ij}^{\theta} :=  E_{ij} + E_{\tau i +1,\tau j+1}, \qquad
E_{ij}^{\theta}(\mathbf{v},a):= \diag(\mathbf{v}) +  aE_{ij}^{\theta},
\end{align} 
where $\tau i:=N-i$ for $1\leq i\leq 2n$.
Then $E_{ij}^{\theta}=E_{\tau i+1,\tau j+1}^{\theta}$. Let $\mathbf{e}_i$ ($1\leq i\leq N$) be the standard basis for $\C^N$ (viewed as row vectors). From definition, we get
\begin{align*}
		\mathcal{O}_{E_{i,i+1}^{\theta}(\mathbf{v},a)}&=\bigg\{(F,F')\mid
		F=(V_k)_{0\le k \le N} 
			, F'=(V'_k)_{0\le k \le N}, V'_i\stackrel{a}{\subset} V_i, V_k=V'_k \textit{ if } k\neq i,\tau i\bigg\}\\&\subset \mathscr{F}_{\mathbf{v}+a\mathbf{e}_i+a\mathbf{e}_{\tau i+1} }\times \mathscr{F}_{\mathbf{v}+a\mathbf{e}_{i+1}+a\mathbf{e}_{\tau i} },
\end{align*}
and it is a closed orbit. Here $V'_i\stackrel{a}{\subset} V_i$ means that $V_i'$ is a vector subspace in $V_i$ of codimension $a$.

\subsection{Steinberg variety}
The (generalized) {\em Steinberg variety} $Z$ is defined as
$$Z:=T^*\mathscr{F}\times_{\mathcal{N}}T^*\mathscr{F}=\bigsqcup_{\mathbf{v},\mathbf{w}\in\Lambda^\imath_{N,d}}T^*\mathscr{F}_{\mathbf{v}}\times_{\mathcal{N}}T^*\mathscr{F}_{\mathbf{w}},$$
 where $\mathcal{N}$ is the nilpotent cone of $\mathrm{Lie}(G)$.
Denote $Z_{\mathbf{v},\mathbf{w}}:=T^*\mathscr{F}_{\mathbf{v}}\times_{\mathcal{N}}T^*\mathscr{F}_{\mathbf{w}} $. Recalling that the diagonal $G$-orbits on $\mathscr{F}_{\mathbf{v}}\times \mathscr{F}_{\mathbf{w}}$ are indexed by $\Xi_{N,d}(\mathbf{v},\mathbf{w})$, we have 
\[Z_{\mathbf{v},\mathbf{w}}=\bigsqcup_{A\in \Xi_{N,d}(\mathbf{v},\mathbf{w})}Z_A,\]
where $Z_A:=T^*_{\mathcal{O}_A}(\mathscr{F}_{\mathbf{v}}\times \mathscr{F}_{\mathbf{w}})$ is the conormal bundle of the $G$-orbit corresponding to $A\in \Xi_{N,d}(\mathbf{v},\mathbf{w})$. In particular, if $A=\diag(\mathbf{v})$ for some $\mathbf{v}\in \Xi_{N,d}$, then $Z_{\diag(\mathbf{v})}$ is the diagonal copy of $T^*\mathscr{F}_\mathbf{v}$ inside $Z$.

Via convolution, the algebra $K^{G\times\mathbb{C}^{*} }(Z)$ acts on $K^{G\times\mathbb{C}^{*} }(T^*\mathscr{F})$. By arguments similar to \cite[Claim~7.6.7]{CG97}, we have
\begin{lemma}\label{lemma:faithful}
This convolution action gives a faithful representation of $K^{G\times\mathbb{C}^{*} }(Z)$ on $K^{G\times\mathbb{C}^{*} }(T^*\mathscr{F})$.
\end{lemma}

By the same argument as in \cite[Theorem 2.10]{SY26}, we get the following result.
\begin{prop}\label{prop:generators}
    The convolution algebra $K^{G\times\bbC^*}(\calZ)\otimes_{\mathbb{C}[q,q^{-1}]}\mathbb{C}(q)$ is generated by $K^{G\times\bbC^*}(\calZ_{\diag(\bfv)})$ for $\bfv\in \Lambda_{N,d}^\imath$, and $K^{G\times\bbC^*}(\calZ_{E^\theta_{i,i+1}(\bfv',1)})$ for $\bfv'\in \Lambda_{N,d-1}^\imath$ and $1\leq i\leq N-1$.
\end{prop}

\section{The correspondences and operators}\label{sec:operators}
In this section, we introduce some special elements in the equivariant K-theory of the Steinberg variety and write down the explicit formula for them under the faithful representation in Lemma \ref{lemma:faithful}.

\subsection{The coordinate ring $\mathbf{R}$}
We shall always denote, for each $1\leq r\leq 2d$,
$$r':=2d+1-r.$$

Let $T\subset B$ be a maximal torus and a Borel subgroup of $G$, respectively. Let $x_1,x_2,\ldots, x_{2d}$ be the standard dual coordinates on $T$, so that $x_{r'}=x_r^{-1}$ for any $1\leq r\leq 2d$. 

Let 
$$\mathbf{R} =\mathbb{C}[x_1,x_2,\ldots,x_{2d}] =\mathbb{C}[x_1^{\pm 1},x_2^{\pm 1},\ldots,x_d^{\pm 1}]\simeq K^G(G/B),$$
which admits a natural action of $W_{\mf c}=\mathbb{Z}_2^{d}\rtimes \mathfrak{S}_{d}$ as follows. 
For any permutation $\sigma \in \mathfrak{S}_{d}$, the action of $\sigma$ on $\mbf R$ is given by
$$\sigma: \mathbf{R} \rightarrow \mathbf{R},~~f(x_1^{\pm 1},x_2^{\pm 1},\cdots,x_d^{\pm 1}) \mapsto  f(x_{\sigma(1)}^{\pm 1},x_{\sigma(2)}^{\pm 1},\cdots,x_{\sigma(d)}^{\pm 1}).$$
For any $m\in [1,d]$, 
let $\iota_m$ denote the nontrivial element in the $m$-th copy of $\bbZ_2$ in $\bbZ_2^d$. The action of  $\iota_m$ on $\mathbf{R}$ is defined by
\begin{equation*}
  \begin{split}
\iota_m: \mathbf{R} &\rightarrow \mathbf{R},\\
 f(x_1^{\pm 1},\cdots,x_{m-1}^{\pm 1}, x_m^{\pm 1},x_{m+1}^{\pm 1},\cdots x_d^{\pm 1}) &\mapsto  f(x_1^{\pm 1},\cdots,x_{m-1}^{\pm 1}, x_m^{\mp 1},x_{m+1}^{\pm 1},\cdots x_d^{\pm 1}).
  \end{split}
\end{equation*}
For $\mathbf{v}\in\Lambda^\imath_{N,d}$ (resp. $A\in\Xi_{N,d}$), we use $\mbf R^{[\mbf v]}$ (resp. $\mbf R^{[A]}$) to denote the invariant subalgebra $\mbf R^{W_{[\mbf v]}}$  (resp. $\mbf R^{W_{[A]}}$) for simplicity.

For two subgroups of the Weyl group $W_1\subset W_2\subset W_\fc$, we define a map
 $$W_2/W_1 : \mathbf{R}^{W_1} \rightarrow \mathbf{R}^{W_2}, \quad f \mapsto \sum\limits_{\sigma \in W_2/W_1}\sigma(f).$$

The following result can be proved exactly as the same as \cite[Proposition 2.4]{SW24}.
\begin{prop}\label{prop:pushforward}
Let $\mathbf{v}, \mathbf{v}_1, \mathbf{v}_2 \in \Lambda_{N,d}^\imath$ and $A \in \Xi_{N,d}(\mathbf{v}_1,\mathbf{v}_2)$.
\begin{itemize}
\item[(a)] There exist $\mathbb{C}$-algebra isomorphisms $K^{G}(\mathscr{F}_{\mathbf{v}}) \simeq \mathbf{R}^{[\mathbf{v}]} \text{and} \ K^{G}(\mathcal{O}_A) \simeq \mathbf{R}^{[A]}$.
\item[(b)] The first projection map $p_{1,A}: \mathcal{O}_A \rightarrow \mathscr{F}_{\mathbf{v}_1}$ is a smooth fibration.
Moreover, if $\mathcal{O}_A $ is closed, then the direct image morphism $Rp_{{1,A} *}$ is given by
$$Rp_{{1,A} *}[\mathscr{F}] =  W_{[\mathbf{v}_1]}/W_{[A]}\bigg(\frac{[\mathscr{F}]}{\bigwedge(T_{p_{1,A}}^*)}\bigg),$$
where $T_{p_{1,A}}^*$ is the relative cotangent bundle and $\bigwedge(T_{p_{1,A}}^*) = \sum_i(-1)^i\bigwedge^iT^*_{p_{1,A}}$.
\end{itemize}
\end{prop}

Recall $\tau$ is the involution in Figure~\ref{fig:AIII2n}, so that $\tau s=N-s$ for any $1\leq s<N$. 
For any partition $I=(I_1,I_2,\cdots, I_N)$ of $\{1,2,\dots,2d\}$ satisfying $r\in I_s$ if and only if $r'\in I_{N+1-s}=I_{\tau s+1}$, let $(x_I)$ denote the variables 
\[(x_{i_{1,1}},\dots, x_{i_{1},j_1},x_{i_{2,1}},\dots, x_{i_{2},j_2},\dots, x_{i_{N},1},\dots,x_{i_N,j_N}),\] where $I_k=\{i_{k,1},i_{k,2},\dots,i_{k,j_k}\}$. For any $r\in I_s$, let $\tau_r^+I$ be the partition of $\{1,2,\dots,2d\}$ with $r$ shifted from $I_s$ to $I_{s+1}$ and $r'$ shifted from $I_{\tau s+1}$ to $I_{\tau s}$. 
For example, let us consider the case $d=3$ and $N=3$. Let $I=(\{1,2\},\{3,4\},\{5,6\})$ be a partition of $\{1,2,3,4,5,6\}$. Then \[f(x_{\tau_1^+I})=f(x_2,x_1,x_3,x_4,x_6,x_5).\]

\subsection{Combinatorial operators}\label{sec:comboper}
For any $\mathbf{v}\in\Lambda^\imath_{N,d}$, we use $f(x_{[\mathbf{v}]})$ to denote an element in $K^{G\times \C^*}(\mathscr{F}_{\mathbf{v}}) \simeq\mathbf{R}^{[\mathbf{v}]}[q,q^{-1}]$. Let 
\[\mathbf{P}:=\bigoplus_{\mathbf{v}\in \Lambda^\imath_{N,d}}\mathbf{R}^{[\mathbf{v}]}[q,q^{-1}]\simeq K^{G\times \C^*}(\mathscr{F})\simeq K^{G\times\mathbb{C}^*}(T^*\mathscr{F}).\]
In this subsection, we will define some explicit operators on $\mathbf{P}$.

Let $\theta_m(z):=\frac{q^mz-1}{z-q^m}$. Then $\theta_1(z)^{-1}=\theta_1(z^{-1})$. For any subset $S\subset [d]$, let $\Phi_S(z):=\prod_{t\in S}\theta_1(z/x_t)$. Let $\delta(z)=\sum_{k\in \bbZ}z^k$.

For $1\leq i\leq 2n$, define the operator $\hat{\K}_i$, which acts on $\bfR^{[\bfv]}[q,q^{-1}]$ by the scalar $q^{-v_i+v_{i+1}+\delta_{i,n}}$.  

For $1\leq i\leq 2n$, let $\hat{\Theta}_{i}(z)\in \End(\mathbf{P})[[z]]$ be the operator on $\mathbf{P}$, whose coefficients acts on $\mathbf{R}^{[\mathbf{v}]}$ by multiplying by the coefficients of the expansion at $z=0$ of the following rational function 
\begin{align*}
        \hat{\Theta}_{i,\mathbf{v}}(z):=&q^{v_{i+1}-v_i}\cdot \Phi_{[\mathbf{v}]_{i}}(q^{1-i}z^{-1})\cdot \Phi_{[\mathbf{v}]_{i+1}}(q^{-1-i}z^{-1})^{-1}\cdot \bigg(\frac{1-q^{2n}z^2}{1-q^{2n+2}z^2}\bigg)^{\delta_{i,n+1}}.
\end{align*}
Consider $\hat{\Theta}_{i,\mathbf{v}}(z)$ as a series in $z$, then it has constant $1$. 
For $r>0$, let $(v-v^{-1})\hat{\Theta}_{i,\mathbf{v},r}$ be the coefficient of $z^r$ of $\hat{\Theta}_{i,\mathbf{v}}(z)$, and let 
\[\hat{\Theta}_{i,r}:=\sum_{\bfv}\hat{\Theta}_{i,\mathbf{v},r}\]
By the construction, $\hat{\Theta}_{i,r}\in \bfP\simeq K^{G\times \mathbb{C}^*}(T^*\mathscr{F})$. Since $T^*\mathscr{F}$ embeds diagonally into the Steinberg variety $Z$, we can think of $\hat{\Theta}_{i,r}$ as a class in 
$K^{G\times \mathbb{C}^*}(Z)$.

For $1\leq i\leq 2n$, define the operators $\hat{B}_{i,r}$ on $\mathbf{P}$ by the following formula
\begin{equation}\label{equ:Bope}
    (\hat{B}_{i,r}f)(x_{[\mathbf{v}]}):=\sum_{j\in [\bfv]_i}x_{j}^r\cdot \Phi_{[\bfv]_i\setminus\{j\}}(qx_j) \cdot  f(x_{\tau^+_j[\mathbf{v}]}),
\end{equation}
where $f\in \mathbf{R}^{[\mathbf{v}']}[q,q^{-1}]$ with $\mathbf{v}':=\mathbf{v}-\mathbf{e}_i+\mathbf{e}_{i+1}+\mathbf{e}_{\tau i}-\mathbf{e}_{\tau i+1}.$
Let $\hat{B}_{i}(z):=\sum_{r\in \mathbb{Z}}q^{ri}\hat{B}_{i,r} z^r$. Then
\begin{equation*}
    (\hat{B}_{i}(z)f)(x_{[\mathbf{v}]})=\sum_{j\in [\bfv]_i}\delta(q^izx_j)\cdot \Phi_{[\bfv]_i\setminus\{j\}}(qx_j) \cdot  f(x_{\tau^+_j[\mathbf{v}]}).
\end{equation*}

\subsection{Geometric $B$-operators}\label{sec:GeoB}
For any $\mathbf{v}\in \Xi_{N,d}$ and any $1\leq i\leq 2n$, let 
\begin{equation}\label{equ:v'v''}
    \mathbf{v}'':=\mathbf{v}-\mathbf{e}_i-\mathbf{e}_{\tau i+1},\textit{ and }\mathbf{v}':=\mathbf{v}-\mathbf{e}_i+\mathbf{e}_{i+1}+\mathbf{e}_{\tau i}-\mathbf{e}_{\tau i+1}.
\end{equation}
The matrix $E^\theta_{i,i+1}(\mathbf{v}'',1)$ defined in \eqref{eq:Eijv} is minimal in the order \eqref{equ:order}. Hence, the $G$-orbit $\mathcal{O}_{E^\theta_{i,i+1}(\mathbf{v}'',1)}$ in $\mathscr{F} \times \mathscr{F}$ is closed, and it is given by 
\begin{align*}
\mathcal{O}_{E_{i,i+1}^{\theta}(\mathbf{v}'',1)}&=\bigg\{(F,F')\mid 
\substack{F=(V_k)_{0\le k \le N} 
\in \mathscr{F}_{\mathbf{v} }\\ F'=(V'_k)_{0\le k \le N}
\in \mathscr{F}_{\mathbf{v}' }}, V'_i\stackrel{1}{\subset} V_i, V_k=V'_k \textit{ if } k\neq i,\tau i\bigg\}.
\end{align*}
Let $\mathcal{L}_{\mathbf{v},i}$ denote the tautological line bundle on $\mathcal{O}_{E_{i,i+1}^{\theta}(\mathbf{v}'',1)}$ whose stalk at a point $(F,F')$ above is given by $V_i/V'_i$. Under the isomorphism $K^{G}(\mathcal{O}_{E_{i,i+1}^{\theta}(\mathbf{v}'',1)})\simeq \mathbf{R}^{W_{[E_{i,i+1}^{\theta}(\mathbf{v}'',1)]}}$ shown in Proposition \ref{prop:pushforward}(a), this line bundle $\mathcal{L}_{\mathbf{v},i}$ corresponds to $x_{\bar{v}_{i}}\in \mathbf{R}^{W_{[E_{i,i+1}^{\theta}(\mathbf{v}'',1)]}}$.

We denote by $p_1$ and $p_2$ the two projections from the orbit $\mathcal{O}_{E_{i,i+1}^{\theta}(\mathbf{v}'',1)}$ to $\mathscr{F}_{\mathbf{v}}$ and $\mathscr{F}_{\mathbf{v}'}$, respectively. For $i\neq n+1$, the fiber $p_1^{-1}(F)$ is the Grassmannian $\Gr(v_i-1,V_{i}/V_{i-1})$. Hence,
\begin{equation*}
    T^*_{p_1}=\sum_{\bar{v}_{i-1}<t< \bar{v}_i}x_t/x_{\bar{v}_i}\in K^{G}(\mathcal{O}_{E_{i,i+1}^{\theta}(\mathbf{v}'',1)})\simeq \mathbf{R}^{[E_{i,i+1}^{\theta}(\mathbf{v}'',1)]}.
\end{equation*}
For $i=n+1$, the fiber $p_1^{-1}(F)$ is 
\begin{align*}
    \bigg\{ F'= (V'_k)_{0\le k \le N}\in \mathscr{F}_{\mathbf{v}'}\mid V_n\stackrel{1}{\subset} V'_n\subset V'_{n+1}=(V'_n)^\perp\subset V_{n+1}=V_n^\perp,&\\
    V_k=V'_k \textit{ if } k\neq n, n+1 & \bigg\}.
\end{align*}
Therefore,
\begin{align}\label{equ:Eu2}
    T^*_{p_1}
    =&\sum_{\bar{v}_n+1 < t \leq \bar{v}_{n+1}}x_{\bar{v}_{n}+1}/x_t\in K^G(\mathcal{O}_{E_{n+1,n+2}^{\theta}(\mathbf{v}'',1)})
    \simeq \mathbf{R}^{[E_{n+1,n+2}^{\theta}(\mathbf{v}'',1)]}.
\end{align}

Recall that $Z_{E_{i,i+1}^{\theta}(\mathbf{v}'',1)}\subset Z$ is the conormal bundle of the $G$-orbit $\mathcal{O}_{E_{i,i+1}^{\theta}(\mathbf{v}'',1)}$ inside $\mathscr{F}\times \mathscr{F}$.
Let $\pi: Z_{E_{i,i+1}^{\theta}(\mathbf{v}'',1)} \rightarrow \mathcal{O}_{E_{i,i+ 1}^{\theta}(\mathbf{v}'',1)}$ be the projection. For any $1\leq i\leq n$ and $r\in \bbZ$, define
\begin{equation*}
\mathscr{B}_{i,\mathbf{v},r} =  \pi^{*}(\mathrm{Det}(T^{*}_{p_1})\otimes \mathcal{L}_{\mathbf{v},i}^{\otimes r})\in K^{G\times \bbC^*}(Z_{E_{i,i+1}^{\theta}(\mathbf{v}'',1)})
\end{equation*}
and 
\begin{equation*}
\mathscr{B}_{i,r}= \sum\limits_{\mathbf{v}} (-q)^{1-v_i}\mathscr{B}_{i,\mathbf{v},r}\in K^{G\times \bbC^*}(Z).
\end{equation*}

The following proposition says that $\mathscr{B}_{i,r}$ geometrizes the operator $\hat{B}_{i,r}$ in \eqref{equ:Bope}. Recall in Lemma~\ref{lemma:faithful} that $K^{G\times \mathbb{C}^*}(Z)$ acts faithfully on $K^{G\times \C^*}(\mathscr{F})$.

\begin{prop}\label{prop:Bir}
    Under the isomorphism in Proposition \ref{prop:pushforward}, the convolution action of $\mathscr{B}_{i,r}\in K^{G\times \mathbb{C}^*}(Z)$ on $K^{G\times \C^*}(\mathscr{F})\simeq \mathbf{P}$ is given by the above operator $\hat{B}_{i,r}$ in \eqref{equ:Bope}.
\end{prop}
\begin{proof}
Let us first assume $i\neq n+1$. For any $f \in  K^{\mathbb{C}^{*}\times G}(\mathscr{F}_{\mathbf{v}'})\simeq \bfR^{W_{[\bfv']}}[q,q^{-1}]$, $\mathscr{B}_{i,r}\star f=\mathscr{B}_{i,\mathbf{v},r}\star f$. Moreover,
\begin{align*}
&(-q)^{1-v_i}\mathscr{B}_{i,\mathbf{v},r}\star f\\
=& (-q)^{1-v_i}Rp_{1*}\bigg(\bigwedge\nolimits_{q^2}T_{p_1}\otimes p_2^*f\otimes \pi^{*}(\mathrm{Det}(T^{*}_{p_1})\otimes \mathcal{L}_{\mathbf{v},i}^{\otimes r})\bigg)\\
=&(-q)^{1-v_i}W_{[\mathbf{v}]}/W_{[E_{i,i+1}^{\theta}(\mathbf{v}'',1)]}\bigg(\frac{\bigwedge\nolimits_{q^2}T_{p_1}\otimes p_2^*f\otimes \pi^{*}(\mathrm{Det}(T^{*}_{p_1})\otimes \mathcal{L}_{\mathbf{v},i}^{\otimes r})}{\bigwedge (T^*_{p_1})}\bigg)\\
=&\mathfrak{S}_{[\bar{v}_{i-1}+1,\bar{v}_i]}/\mathfrak{S}_{[\bar{v}_{i-1}+1,\bar{v}_i-1]}\bigg(x_{\bar{v}_{i}}^r\prod_{\bar{v}_{i-1} < t \leq \bar{v}_{i}-1} \frac{q-q^{-1}x_{t}/x_{\bar{v}_{i}}}{1-x_t/x_{\bar{v}_{i}}} \cdot f\bigg)\\
=&\sum_{j\in [\bfv]_i}(j,\bar{v}_i)\bigg(x_{\bar{v}_{i}}^r\Phi_{[\bfv]_i\setminus\{\bar{v}_i\}}(qx_{\bar{v}_i}) \cdot f\bigg)\\
=&\sum_{j\in [\bfv]_i}x_{j}^r\Phi_{[\bfv]_i\setminus\{j\}}(qx_j) \cdot (j,\bar{v}_i)f.
\end{align*}
Here the first equality follows from Lemma \ref{lem:convolution}, the second one follows from Proposition \ref{prop:pushforward}(b), and the third one follows from $T^*_{p_{1}}=\sum_{ \bar{v}_{i-1} < t \leq  \bar{v}_{i}-1} x_t/x_{\bar{v}_{i}}$. Notice that for any $j\in [\bfv]_i$ and $f\in \bfR^{W_{[\bfv']}}[q,q^{-1}]$,
\[((j,\bar{v}_i)f)(x_{[\bfv]})=f(x_{\tau^+_j[\bfv]}).\]
This finishes the proof when $i\neq n+1$.

Now let us assume $i=n+1$. By definition, $E_{n+1,n+2}^{\theta}(\mathbf{v}'',1)=E_{n+1,n}^{\theta}(\mathbf{v}'',1)$. Hence, by the same argument as above, 
    \begin{align*}
        &(\mathscr{B}_{n+1,r}\star f)(x_{[\mathbf{v}]})\\
        =&(\mathscr{B}_{n+1,\mathbf{v},r}\star f)(x_{[\mathbf{v}]})\\=&(-q)^{1-v_{n+1}}W_{[\mathbf{v}]}/W_{[E_{n+1,n}^{\theta}(\mathbf{v}'',1)]}\bigg(\frac{\bigwedge\nolimits_{q^2}T_{p_1}\otimes p_2^*f\otimes \pi^{*}(\mathrm{Det}(T^{*}_{p_1})\otimes \mathcal{L}_{\mathbf{v},i}^{\otimes r})}{\bigwedge (T^*_{p_1})}\bigg)\\
        =&\sum_{j\in [\mathbf{v}]_{n+1}^\mathfrak{c}}(j,1+\bar{v}_n)\bigg(\prod_{\bar{v}_n+1<t\leq \bar{v}_{n+1}}\frac{q-q^{-1}x_{\bar{v}_n+1}/x_t}{1-x_{\bar{v}_n+1}/x_t}\cdot x_{1+\bar{v}_{n}}^{-r}\cdot f\bigg)(x_{[\mathbf{v}]})\\
        &+\sum_{j\in [\mathbf{v}]_{n+1}^\mathfrak{c}}\iota_j(j,1+\bar{v}_n)\bigg(\prod_{\bar{v}_n+1<t\leq \bar{v}_{n+1}}\frac{q-q^{-1}x_{\bar{v}_n+1}/x_t}{1-x_{\bar{v}_n+1}/x_t}\cdot x_{1+\bar{v}_{n}}^{-r}\cdot f\bigg)(x_{[\mathbf{v}]})\\
        =&\sum_{j\in [\mathbf{v}]_{n+1}}\prod_{\bar{v}_n+1\leq t\leq \bar{v}_{n+1},t\neq j}\frac{q-q^{-1}x_j/x_t}{1-x_j/x_t}\cdot x_j^{-r}\cdot f(x_{\tau_j^-[\mathbf{v}]})\\
        =&\sum_{j\in [\mathbf{v}]_{n+1}}\Phi_{[\mathbf{v}]_{n+1}\setminus\{j\}}(qx_j)\cdot x_j^{r}\cdot f(x_{\tau_j^+[\mathbf{v}]}).
    \end{align*}
    Here the third equality follows from \eqref{equ:Eu2}, $x_{1+\bar{v}_n}^{-1}=x_{\bar{v}_{n+1}}$, and the fact that a set of representatives for $W_{[\mathbf{v}]}/W_{[E_{n+1,n}^{\theta}(\mathbf{v}'',1)]}$ is 
    \[\{(1+\bar{v}_n,j),\iota_{j}(1+\bar{v}_n,j))\mid 1+\bar{v}_n\leq j\leq d\}.\]
    The forth one follows from the change of variable $j\mapsto 2d+1-j=j'$ for the second summand, and  
    \[((\bar{v}_n+1,j)f)(x_{[\mathbf{v}]})=f(x_{\tau^-_j[\mathbf{v}]}) \quad \textit{and} \quad (\iota_j(\bar{v}_n+1,j)f)(x_{[\mathbf{v}]})=f(x_{\tau_{2d+1-j}^-[\mathbf{v}]})\]
    for any $j\in [\mathbf{v}]_{n+1}^\mathfrak{c}$, $f\in \mathbf{R}^{[\mathbf{v}']}[q,q^{-1}]$. Finally, the last one follows from the change of variables $j\mapsto j'$ and $t\mapsto t'$.
    This finishes the proof.
\end{proof}

\section{K-theoretic realization of $\widetilde{\mathbf{U}}^\imath$}
\label{sec:main}
In this section, we state the main results and provide detail proofs. 

\subsection{Drinfeld new presentation}\label{sec:drinfeld}
Firstly, we recall the Drinfeld new presentation for the affine iquantum group $\widetilde{\mathbf{U}}^\imath$ associated to the Satake diagram drawn in Figure~\ref{fig:AIII2n}. 

Let $(c_{ij})_{0\leq i,j\leq 2n}$ be the Cartan matrix of affine type $\tilde{A}_{2n}$,
Denote by $[k]$ and $\qbinom{k}{r}$ the $v$-integers and $v$-binomials, for $k,r\in \mbb N$. 

The (universal) quasi-split affine iquantum group of type AIII$_{2n}^{(\tau)}$ is the $\C(v)$-algebra $\tUi $ generated by $B_i$, $\K_i^{\pm 1}$ $(0\leq i\leq 2n)$, subject to the following relations (see \cite{CLW21}):
	\begin{align*}
		\K_i\K_i^{-1} =\K_i^{-1}\K_i=1,  \quad \K_i\K_\ell &=\K_\ell \K_i, \quad
		\K_\ell B_i=v^{c_{\tau \ell,i} -c_{\ell i}} B_i \K_\ell, \\
		B_iB_j -B_j B_i&=0, \qquad\qquad\qquad \text{ if } c_{ij}=0 \text{ and }\tau i\neq j, \\
		\sum_{s=0}^{1-c_{ij}} (-1)^s \qbinom{1-c_{ij}}{s} & B_i^{s}B_jB_i^{1-c_{ij}-s} =0, \quad \text{ if } j\neq \tau i\neq i ,
		\\
		B_{\tau i}B_i -B_i B_{\tau i}& =   \frac{\K_i -\K_{\tau i}}{v-v^{-1}},
		\quad \text{ if }  c_{i,\tau i}=0,	\\
		B_i^2 B_j -[2] B_i B_j B_i +B_j B_i^2 &= - v^{-1}  B_j \K_i,  \quad \text{if }c_{ij}=-1 \text{ and }c_{i,\tau i}=2,\\
		B_i^2B_{\tau i}-[2]B_iB_{\tau i}B_i+B_{\tau i}B_i^2 &=-[2](v\K_iB_i+vB_i\K_{\tau i}),\quad \textit{ if } c_{i,\tau i}=-1.
	\end{align*}
 
We need to use the Drinfeld new presentation for $\tUi$, which is recently established in \cite{LPWZ25}. It has generators $B_{il},\Theta_{im}, \K_i^{\pm 1}$, and $C^{\pm 1}$ for $1\leq i\leq 2n, l\in \mathbb{Z}$ and $m>0$. We introduce the following generating functions for an indeterminate $z$: 
\begin{equation}\label{def:genfunc}
\bB_{i}(z)  =\sum_{k\in\bbZ} B_{i,k}z^{k}, \quad \bTh_{i}(z) =1+ \sum_{k\ge 1} (q-q^{-1})\Theta_{i,k}z^{k},\quad \bDel(z)  =\sum_{k\in\bbZ} C^k z^k.
\end{equation}
Define 
\begin{align}\label{def:sym}
	&\mathbb{S}_{i,j}(w_1,w_2| z):=\\
	& \Sym_{w_1,w_2}\bigg(\bB_i(w_1)\bB_i(w_2)\bB_j(z)-[2]\bB_i(w_1)\bB_j(z)\bB_i(w_2)+\bB_j(z)\bB_i(w_1)\bB_i(w_2)\bigg),\nonumber
\end{align}
where it is understood that $\Sym_{w_1,w_2} f(w_1,w_2) := f(w_1,w_2) +f(w_2,w_1).$

\begin{theorem}\cite[Theorem 4.9]{LPWZ25}  \label{thm:genfun}
The quasi-split affine iquantum group $\tUi$ is isomorphic to the
$\C(v)$-algebra generated by the elements $B_{il}$, $\Theta_{im}$, $\K_i^{\pm 1}$, and $C^{\pm 1}$, where $1\leq i\leq 2n$, $l\in\Z$ and $m>0$, subject to the following relations:
\begin{align}\label{rel:1}
C \textit{ is central}, \quad [\K_i,\K_j]=[\K_i, \bTh_j(w)] = 0,\quad \K_i\bB_{j}(z) =v^{c_{\tau i,j}-c_{ij}} \bB_{j}(z) \K_i,
\end{align}
\begin{align}
[\bTh_{i}(z), \bTh_{j}(w)] =0,
\end{align}
\begin{align}
\bB_j(w)  \bTh_i(z)
 =  \frac{(1 -v^{c_{ij}}zw^{-1}) (1 -v^{-c_{\tau i,j}}zw C)}{(1 -v^{-c_{ij}}zw^{-1})(1 -v^{c_{\tau i,j}} zw C)}
 \bTh_i(z) \bB_j(w), 
\end{align}
\begin{align}
[\bB_i(z), \bB_{\tau i}(w)] 
= \frac{\bDel (zw)}{v-v^{-1}} (\K_{\tau i} \bTh_i(z) -\K_{i} \bTh_{\tau i}(w)),  \qquad \text{ if }c_{i,\tau i}=0,
\end{align}
\begin{align}\label{rel:6}
(v^{c_{ij}}z -w) \bB_i(z) \bB_j(w) +(v^{c_{ij}}w-z) \bB_j(w) \bB_i(z)=0, \quad \text{ if } j \neq  \tau i,
\end{align}
\begin{align}\label{equ:bbnn+1}
&(v^{-1}z-w) \bB_i(z) \bB_{\tau i} (w) +(v^{-1}w-z) \bB_{\tau i}(w) \bB_i(z)
\\ 
&= \frac{\bDel(zw) }{1-v^2} 
 \big((z -vw) \K_{{i}} \bTh_{\tau i} (w) + (w -vz) \K_{\tau i} \bTh_i(z) \big), \quad \text{ if }  c_{i,\tau i}=-1,\notag
\end{align}
\begin{align}\label{rel:8}
\bB_i(w)\bB_j(z)=\bB_j(z)\bB_i(w), \qquad\text{ if }c_{ij}=0 \text{ and }\tau i\neq j, 
\end{align}
\begin{align}\label{rel:9}
\mathbb{S}_{i,j}(w_1,w_2\mid z)=0,\quad \textit{ if }c_{ij}=-1,j\neq \tau i\neq i,
\end{align}
\begin{align}\label{equ:iserre}
& (v-v^{-1})\mathbb{S}_{i,\tau i}(w_1,w_2\mid z)\\
=&-v^{-1}[2]\Sym_{w_1,w_2}\bDel(w_2z)\frac{1-vw_2^{-1}z}{1-v^{-2}w_1w_2^{-1}}\bB_i(w_1)\bTh_{\tau i}(z)\K_i \notag\\
	&+[2]\Sym_{w_1,w_2}\bDel(w_2z)\frac{1-vw_2^{-1}z}{1-v^2w_1w_2^{-1}}\bTh_{\tau i}(z)\K_{i} \bB_i(w_1) \notag\\
	&+v[2]\Sym_{w_1,w_2}\bDel(w_2z)\frac{w_1^{-1}z-vw_1^{-1}w_2}{1-v^2w_1^{-1}w_2}\bB_i(w_1)\bTh_{i}(w_2)\K_{\tau i}\notag\\
	&+v^{-2}[2]\Sym_{w_1,w_2}\bDel(w_2z)\frac{vw_1^{-1}w_2-zw_1^{-1}}{1-v^{-2}w_1^{-1}w_2}\bTh_{i}(w_2)\K_{\tau i}\bB_i(w_1),\quad \textit{if }c_{i,\tau i}=-1.\notag
 \end{align}
\end{theorem}

\begin{remark}\label{rem:iserre}
    The Serre relation \eqref{equ:iserre} can be simplified as follows. Since $c_{i,\tau i}=-1$, we get
    \[\bB_i(w)  \bTh_i(z)=  \frac{(1 -v^2zw^{-1}) (1 -vzw C)}{(1 -v^{-2}zw^{-1})(1 -v^{-1} zw C)}\bTh_i(z) \bB_i(w),\]
and 
\[\bB_i(w)  \bTh_{\tau i}(z)
=  \frac{(1 -v^{-1}zw^{-1}) (1 -v^{-2}zw C)}{(1 -vzw^{-1})(1 -v^{2} zw C)}
\bTh_{\tau i}(z) \bB_i(w).\]
Hence,
\begin{align*}
	&(v-v^{-1})\mathbb{S}_{i,\tau i}(w_1,w_2\mid z)\\
	=&-v^{-1}[2]\Sym_{w_1,w_2}\bDel(w_2z)\frac{1-vw_2^{-1}z}{1-v^{-2}w_1w_2^{-1}} \bB_i(w_1)\bTh_{\tau i}(z)\K_i\\
	&+v^{-3}[2]\Sym_{w_1,w_2}\bDel(w_2z)\frac{1-vw_2^{-1}z}{1-v^2w_1w_2^{-1}}\frac{(1 -vzw_1^{-1})(1 -v^{2} zw_1 C)}{(1 -v^{-1}zw_1^{-1}) (1 -v^{-2}zw_1 C)}\bB_i(w_1)\bTh_{\tau i}(z) \K_{i}\\
	&+[2]v\Sym_{w_1,w_2}\bDel(w_2z)\frac{w_1^{-1}z-vw_1^{-1}w_2}{1-v^2w_1^{-1}w_2} \bB_i(w_1)\bTh_{i}(w_2)\K_{\tau i}\\
	&+v[2]\Sym_{w_1,w_2}\bDel(w_2z)\frac{vw_1^{-1}w_2-zw_1^{-1}}{1-v^{-2}w_1^{-1}w_2}\frac{(1 -v^{-2}w_2w_1^{-1})(1 -v^{-1} w_2w_1C)}{(1 -v^2w_2w_1^{-1}) (1 -vw_2w_1 C)}\bB_i(w_1)\bTh_{i}(w_2)\K_{\tau i}\\
	=&[2]\Sym_{w_1,w_2}\bDel(w_2z)\frac{(1-v^2)(w_2-vz)w_1}{(v^2w_2-w_1)(vw_1-z)} \bB_i(w_1)\bTh_{\tau i}(z)\K_i\\
	&+[2]\Sym_{w_1,w_2}\bDel(w_2z)\frac{(vw_2-z)(v^2-1) w_1}{(v^2w_2-w_1) ( vw_1-z)}\bB_i(w_1)\bTh_{i}(w_2)\K_{\tau i}.
\end{align*}

\end{remark}

\subsection{The algebra homomorphism}

Recall the operators $\hat{\K}_i, \hat{\Theta}_i(z)$, and $\hat{B}_i(z)$ from Section \ref{sec:comboper}, and the geometric classes $\hat{\Theta}_{im}$ and $\hat{B}_{i,r}$ from Section \ref{sec:GeoB}.
\begin{theorem}\label{thm:polyrep}
The assignment
\[v\mapsto q,\quad C\mapsto q^N, \quad\K_i \mapsto \hat{\K}_i,\quad \bTh_i(z)\mapsto \hat{\Theta}_i(z),\quad
\bB_{i}(z)\mapsto \hat{B}_i(z)\]
extends to a representation of $\tUi$ on $\mathbf{P}\otimes_{\mathbb{C}[q,q^{-1}]} \mathbb{C}(q)$.
\end{theorem}

We will prove this theorem in Section \ref{sec:checkrelations} below. From this, we get the main result of this note.
\begin{theorem}\label{thm:main}
The assignment 
\[v\mapsto q,\quad C\mapsto q^N, \quad\K_i \mapsto \hat{\K}_i,\quad \Theta_{im}\mapsto \hat{\Theta}_{i,m},\quad
B_{il}\mapsto \hat{B}_{il}\]
extends to a $\mathbb{C}(q)$-algebra homomorphism 
\[\tUi\rightarrow K^{G\times\mathbb{C}^*}(Z)\otimes_{\mathbb{C}[q,q^{-1}]}\mathbb{C}(q).\]
\end{theorem}
\begin{proof}
    This follows from Theorem \ref{thm:polyrep}, Proposition \ref{prop:Bir}, and Lemma \ref{lemma:faithful}.
\end{proof}

\subsection{Surjectivity of the homomorphism $\Psi_a$}
 \label{subsec:surj}

Choose $a:=(s,t)\in T\times\bbC^*$, and let $\tUi_t$ denote the specialization of $\tUi$ at $q=t$. Let $\bbC_a$ denote the one-dimensional module of $K^{G\times\bbC^*}(\pt)$ by evaluation at $a$, and 
\[
K^{G\times\mathbb{C}^{*}}(\calZ)_a:=K^{G\times\mathbb{C}^{*}}(\calZ)\otimes_{K^{G\times\bbC^*}(\pt)}\bbC_a.
\]
The specialization of $\Psi$ at $q=t$ gives us an algebra homomorphism
\begin{align*}  
\Psi_a:\tUi_t \longrightarrow K^{G\times\mathbb{C}^{*}}(\calZ)_a.
\end{align*}

\begin{theorem}  \label{thm:surj}
Suppose that $t$ is not a root of unity. Then 
the homomorphism $\Psi_a$ is surjective.
\end{theorem}

The remainder of this subsection is devoted to the proof of Theorem \ref{thm:surj}. To that end, we consider the specialization of $K^{G\times\mathbb{C}^{*}}(\calZ)$ (and its localization) at $q=t$, denoted by $K^{G\times\mathbb{C}^{*}}(\calZ)_t$, and the specialized morphism $\Psi_t$. By the same argument in \cite[Lemma 7.3]{SW24}, $\mathscr{B}_{i,\bfv,k}$ and $\hat{\Theta}_{i,\bfv,k}$ belong to the image of $\Psi_t$. We will use Proposition \ref{prop:generators} to prove Theorem \ref{thm:surj}. The following lemma deals with the diagonal orbits. 
\begin{lemma}
 \label{lem:diagonal}
    Let $\bfv\in \Lambda_{N,d}^\imath$. Suppose that $t$ is not a root of unity. Then the elements $\hat{\K}_i$, $\hat{\bTh}_{i,\bfv,k}$ $(1\leq i\leq 2n)$ and the elements $\sum_{j=1}^d(x_j^k+x_j^{-k})$, for $k\ge 1$, generate $K^{G\times\bbC^*}(\calZ_{\diag(\bfv)})_t\simeq \bfR^{[\bfv]}$.
\end{lemma}

\begin{proof}
    Define operators $\hat{H}_{i,k}$ by the following formula
    \[
    \hat{\Theta}_{i}(z) =\exp\Big((q-q^{-1})\sum_{k\geq 1}\hat{H}_{i,k}z^k\Big).
    \]
    Then a direct computation shows that the operators $\hat{H}_{i,k}$ acts on $K^{G\times\bbC^*}(\calZ_{\diag(\bfv)})_t\simeq \bfR^{[\bfv]}$ by the following scalar multiplications
    \[
    \hat{H}_{i,k}=
        \frac{1}{k}[k]_tt^{(i-1)k}\Big(\sum_{j\in [\bfv]_i}x_j^k-t^{2k}\sum_{j\in [\bfv]_{i+1}}x_j^k\Big)+\delta_{i,n+1}\delta_{2|k}\frac{q^{(2n+2)k-q^{2nk}}}{k(q-q^{-1})},
    \]
    where $\delta_{2|k}=\begin{cases}
        1,\textit{ if 2 divides k};\\
        0, \textit{ otherwise}.
    \end{cases}$
    Then we use the same computation in the proof of \cite[Lemma 7.4]{SW24} to conclude the proof.
\end{proof}

The following lemma deals with the non-diagonal orbits. Let $\bfv\in \Lambda_{N,d}^\imath$, and $\bfv'':=\bfv-\bfe_{i}-\bfe_{\tau i+1}$ as in \eqref{equ:v'v''}.
\begin{lemma}\label{lem:E_ijorbit}
     For $1\leq i\leq 2n$, $K^{G\times\bbC^*}(\calZ_{E^\theta_{i,i+1}(\bfv'',1)})$ is contained in the algebra generated by $\mathscr{B}_{i,\bfv,k}$, for $k\in \bbZ$, and the classes of sheaves supported on the diagonal orbits.
\end{lemma}

\begin{proof}
First of all, the case when $i\neq n+1$ can be proved exactly the same as in \cite[Lemma 7.5]{SW24}. Let us assume $i=n+1$. The orbit $\calO_{E^\theta_{n+1,n+2}(\bfv'',1)}$ is closed, and by Proposition \ref{prop:pushforward}, we have
    \[
    K^{G\times\bbC^*}(\calZ_{E^\theta_{n+1,n+2}(\bfv'',1)})\simeq \bfR^{[E^\theta_{n+1,n+2}(\bfv'',1)]}[q,q^{-1}],
    \]
    where 
    \[
    W_{[E^\theta_{n+1,n+2}(\bfv'',1)]}=\mathfrak{S}_{[1,\bar{v}_1]}\times \cdots \times \mathfrak{S}_{[1+\bar{v}_{n-1},\bar{v}_n]}\times \mathfrak{S}_{[1+\bar{v}_n,1+\bar{v}_{n}]}\times (\mathbf{Z}_2^{|[2+\bar{v}_n,d]|} \rtimes \mathfrak{S}_{[2+\bar{v}_{n},d]}).
    \]
    Moreover, we have 
    \[K^{G\times\bbC^*}(\calZ_{\diag(\bfv)})\simeq \bfR^{\mathfrak{S}_{[1,\bar{v}_1]}\times \cdots \times \mathfrak{S}_{[1+\bar{v}_{n-1},\bar{v}_{n}]}\times(\mathbf{Z}_2^{|[1+\bar{v}_n,d]|} \rtimes \mathfrak{S}_{[1+\bar{v}_{n},d]})}[q,q^{-1}],\]
    By the definition of $\mathscr{B}_{n+1,\bfv,k}$,, we have 
    \[
    \mathscr{B}_{n+1,\bfv,k}=\prod_{1+\bar{v}_{n}< j\leq \bar{v}_{n+1}}\frac{x_{\bar{v}_n+1}}{x_j}\cdot x_{\bar{v}_{n+1}}^k=x_{\bar{v}_{n+1}}^{k+v_{n+1}}\in K^{G\times\bbC^*}(\calZ_{E^\theta_{n+1,n+2}(\bfv'',1)}).
    \]
    Moreover, by \cite[Proposition 3.3]{SW24}, we get that for any $f\in K^{G\times\bbC^*}(\calZ_{\diag(\bfv)})$,
    \[
    f\star\mathscr{B}_{n+1,\bfv,k}=f\cdot x_{\bar{v}_{n+1}}^{k+v_{n+1}}\in K^{G\times\bbC^*}(\calZ_{E^\theta_{n+1,n+2}(\bfv'',1)}).
    \]
    Then the lemma follows from these facts.
\end{proof}

Now we can finish the proof of Theorem \ref{thm:surj}.

\begin{proof}[Proof of Theorem \ref{thm:surj}]
Upon specialization at $s\in G$, the elements $\sum_{j=1}^d(x_j^k+x_j^{-k})$ in Lemma~\ref{lem:diagonal} specialize to scalars. Now Theorem~\ref{thm:surj} follows from Proposition \ref{prop:generators}, Lemma \ref{lem:diagonal}, and Lemma~ \ref{lem:E_ijorbit}.
\end{proof}

\subsection{Representations of $\tUi$}\label{sec:rep}
Recall $a:=(s,t)\in T\times\bbC^*$ as in \S\ref{subsec:surj} and $t$ is not a root of unity. Let $A$ be the subgroup of $G\times\bbC^*$ generated by $a$. Then $\calZ^A=\calZ^a$, where $\calZ^A$ (respectively, $\calZ^a$) denotes the fixed loci of $\calZ$ under the action of $A$ (respectively, $a$). We have the following chain of algebra isomorphisms
\begin{align*}
K^{G\times\bbC^*}(\calZ)_a &:= K^{G\times\bbC^*}(\calZ)\otimes_{K^{G\times\bbC^*}(\pt)}\bbC_a
\\
 & \simeq  K^{A}(\calZ)\otimes_{R(A)}\bbC_a
	\stackrel{r_a}{\xrightarrow{\sim}} K_\bbC(\calZ^a)
	\stackrel{RR}{\xrightarrow{\sim}} H^{BM}_*(\calZ^a,\bbC).
\end{align*}
Here $H^{BM}_*(\calZ^a,\bbC)$ denotes the Borel--Moore homology of $\calZ^a$, which also has a convolution algebra structure, see \cite[Chapter 2]{CG97}. The first isomorphism follows from \cite[Theorem 6.2.10]{CG97}. The map $r_a$ (respectively, $RR$) is the bivariant localization map from Theorem 5.11.10 (respectively, bivariant Riemann--Roch map from Theorem 5.11.11) \textit{loc. cit.}. All these maps respect the convolution algebra structures. Composing with the surjective algebra homomorphism $\Psi_a$ from Theorem \ref{thm:surj}, we get a surjective algebra homomorphism 
\begin{align}  \label{UtoH}
\tUi_t\twoheadrightarrow  H^{BM}_*(\calZ^a,\bbC).
\end{align}
Therefore, every representation of the convolution algebra $H^{BM}_*(\calZ^a,\bbC)$ pulls back to a representation of $\tUi_t$. Since the homomorphism in \eqref{UtoH} is surjective, the pullbacks of irreducible representations will remain irreducible.

Let $G(s)\subset G$ be the centralizer of $s$. By definition, 
\[\calN^a=\{x\in \calN\mid sxs^{-1}=t^{-2}x\}.\]
Denote $\calM:=T^*\mathscr{F}$.
Let $\calM^a:=\bigsqcup_\bfv \calM_\bfv^a$ be the fixed loci. Then the map $\pi:\calM^a\rightarrow \calN^a$ is $G(s)$-equivariant. The equivariant version of the decomposition theorem gives
\[\pi_*\underline{\bbC}_{\calM^a}=\bigoplus_{k\in \bbZ, \phi=(\calO_\phi\subset \calN^a,\chi_\phi)}L_\phi(k)\otimes \IC_\phi[k].\]
Let $L_\phi=\oplus_k L_\phi(k)$. For any $x\in \calO_\phi$, the pullback via \eqref{UtoH} of the $H^{BM}_*(\calZ^a,\bbC)$-module $H_*(\calM_x)_\phi$ is called a {\em standard module} of $\tUi_t$. We also view $L_\phi$ as a $\tUi_t$-module this way. Then the general construction of representations of convolution algebras gives the following result, see \cite{CG97}.

\begin{theorem}
\label{thm:simplestandard}
    Assume that $t$ is not a root of unity.
    \begin{enumerate}
        \item The nonzero module $L_\phi$ is a simple $\tUi_t$-module.
        \item For any $\phi=(\calO_\phi,\chi_\phi)$ and $\psi=(\calO_\psi,\chi_\psi)$ and $x\in \calO_\phi$, 
        \[[H_*(\calM^a_x)_\phi: L_\psi]=\sum_k\dim H^k(i_x^!\IC_\psi)_\phi.\]
    \end{enumerate}
\end{theorem}

\subsection{Proof of Theorem \ref{thm:polyrep}}\label{sec:checkrelations}
It suffices to show that the corresponding operators satisfy the relations in Theorem \ref{thm:genfun}. All the relations except \eqref{equ:bbnn+1} and \eqref{equ:iserre} can be checked in exactly the same way as in \cite{SW24}. Hence, we only prove those two relations.
Let us first prove the following lemma.
\begin{lemma}\label{lem:Thetan}
    The following identity holds:
    \begin{align*}
    &(z-q^{-1}w)f(x_{[\mathbf{v}]})\bigg(\sum_{r\in [\bfv]_n}\delta(q^nzx_r)\delta(q^{n+1}w/x_r)\cdot \Phi_{[\bfv]_n\setminus\{r\}}(qx_r)\Phi_{[\bfv]_{n+1}}(qx_r^{-1})\\
    &-\sum_{s\in [\bfv]_{n+1}}\delta(q^nz/x_s)\delta(q^{n+1}wx_s)\cdot \Phi_{[\bfv]_n}(qx_s^{-1})\Phi_{[\bfv]_{n+1}\setminus\{s\}}(qx_s)\bigg)\\
    =&\frac{\delta(q^{2n+1}zw)}{1-q^2}\bigg((w-qz)\hat{\K}_{n+1} \hat{\bTh}_n(z)f+(z-qw)\hat{\K}_n \hat{\bTh}_{n+1}(w)f\bigg)(x_{[\bfv]}).
\end{align*}
\end{lemma}
\begin{proof}
    Recall 
    \begin{align*}
    (\hat{\K}_{n+1} \hat{\bTh}_n(z)f)(x_{[\bfv]})=\Phi_{[\bfv]_n}(q^{1-n}z^{-1})\Phi_{[\bfv]_{n+1}}(q^{-1-n}z^{-1})^{-1}\cdot f(x_{[\bfv]}),
    \end{align*}
    and 
    \begin{align*}
        &\delta(q^{2n+1}zw)(\hat{\K}_n \hat{\bTh}_{n+1}(w)f)(x_{[\bfv]})\\
        =&\delta(q^{2n+1}zw)\frac{q-q^{2n+1}w^2}{1-q^{2n+2}w^2}\Phi_{[\bfv]_n}(q^{n+2}w)\Phi_{[\bfv]_{n+1}}(q^nw)^{-1}\cdot f(x_{[\bfv]})\\
        =&\delta(q^{2n+1}zw)\frac{qz-w}{z-qw}\Phi_{[\bfv]_n}(q^{n+2}w)\Phi_{[\bfv]_{n+1}}(q^nw)^{-1}\cdot f(x_{[\bfv]}).
    \end{align*}
    Therefore,
    \begin{align*}
    &(z-q^{-1}w)f(x_{[\mathbf{v}]})\bigg(\sum_{r\in [\bfv]_n}\delta(q^nzx_r)\delta(q^{n+1}w/x_r)\cdot \Phi_{[\bfv]_n\setminus\{r\}}(qx_r)\Phi_{[\bfv]_{n+1}}(qx_r^{-1})\\
    &-\sum_{s\in [\bfv]_{n+1}}\delta(q^nz/x_s)\delta(q^{n+1}wx_s)\cdot \Phi_{[\bfv]_n}(qx_s^{-1})\Phi_{[\bfv]_{n+1}\setminus\{s\}}(qx_s)\bigg)\\
    =&(z-q^{-1}w)\frac{f(x_{[\bfv]})}{q-q^{-1}}\Res'\bigg(\delta(q^nxz)\delta(q^{n+1}wx^{-1})\frac{B(x)}{xA(x)}\bigg)\\
    =&(w-qz)f(x_{[\bfv]})\frac{\delta(q^{2n+1}zw)}{1-q^2}\Bigg(\bigg(\frac{B(x)}{A(x)}\bigg)^+|_{x=q^{-n}z^{-1}}-\bigg(\frac{B(x)}{A(x)}\bigg)^-|_{x=q^{n+1}w}\Bigg)\\
    =&\frac{\delta(q^{2n+1}zw)}{1-q^2}\bigg((w-qz)\hat{\K}_{n+1} \hat{\bTh}_n(z)f+(z-qw)\hat{\K}_n \hat{\bTh}_{n+1}(w)f\bigg)(x_{[\bfv]}).
\end{align*}
Here
\[B(x)=\prod_{r\in [\bfv]_n}(qx-q^{-1}x_r)\prod_{s\in [\bfv]_{n+1}}(q^{-1}x-qx_s^{-1}),\] 
\[A(x)=\prod_{r\in [\bfv]_n}(x-x_r)\prod_{s\in [\bfv]_{n+1}}(x-x_s^{-1}),\] 
$\Res'$ denotes the sum of the residues of $\delta(q^nxz)\delta(q^{n+1}wx^{-1})\frac{B(x)}{xA(x)}$ at the singular points $x_r,x_s^{-1}$ for $r\in [\bfv]_n$ and $s\in [\bfv]_{n+1}$, $\left(\frac{B(x)}{A(x)}\right)^+$ (resp. $\left(\frac{B(x)}{A(x)}\right)^-$) denotes the expansion at $x=\infty$ (resp. $x=0$) of $\frac{B(x)}{A(x)}$, the second equality follows from the residue theorem, and the last equality follows from the above formulae and
\[\frac{B(x)}{A(x)}=\Phi_{[\bfv]_n}(qx)\Phi_{[\bfv]_{n+1}}(qx^{-1})=\Phi_{[\bfv]_n}(qx)\Phi_{[\bfv]_{n+1}}(q^{-1}x)^{-1}.\]
\end{proof}

\subsubsection{Relation \eqref{equ:bbnn+1}}
It suffices to show
\begin{align}\label{equ:BBnn+1}
&(q^{-1}z-w)\hat{B}_n(z) \hat{B}_{n+1}(w)+(q^{-1}w-z)\hat{B}_{n+1}(w) \hat{B}_n(z)\\
=&\frac{\bDel(zw)}{1-q^2}\bigg((z-qw)\hat{\K}_n \hat{\Theta}_{n+1}(w) +(w-qz)\hat{\K}_{n+1} \hat{\Theta}_n(z)\bigg).\notag
\end{align}

By definition,
\begin{align*}
    &(\hat{B}_n(z) \hat{B}_{n+1}(w)f)(x_{[\mathbf{v}]})\\
    =&\sum_{r\in [\bfv]_n}\sum_{s\in [\bfv]_{n+1}\cup\{r,r'\}}\delta(q^nzx_r)\delta(q^{n+1}wx_s)\cdot \Phi_{[\bfv]_n\setminus\{r\}}(qx_r)\Phi_{[\bfv]_{n+1}\cup\{r,r'\}\setminus\{s\}}(qx_s) \cdot  f(x_{\tau^+_s\tau^+_r[\mathbf{v}]}),
\end{align*}
and
\begin{align*}
    &(\hat{B}_{n+1}(w)\hat{B}_n(z) f)(x_{[\mathbf{v}]})\\
    =&\sum_{s\in [\bfv]_{n+1}}\sum_{r\in [\bfv]_n\cup\{s'\}}\delta(q^nzx_r)\delta(q^{n+1}wx_s)\cdot \Phi_{[\bfv]_n\cup\{s'\}\setminus\{r\}}(qx_r)\Phi_{[\bfv]_{n+1}\setminus\{s\}}(qx_s) \cdot  f(x_{\tau^+_r\tau^+_s[\mathbf{v}]}).
\end{align*}

Therefore, for $r\in [\bfv]_n$ and $s\in [\bfv]_{n+1}$,
the coefficient of $\Phi_{[\bfv]_n\setminus\{r\}}(qx_r)\Phi_{[\bfv]_{n+1}\setminus\{s\}}(qx_s)\cdot f(x_{\tau_r^+\tau_s^+[\bfv]})$ in the left hand side of \eqref{equ:BBnn+1} is
\[
\delta(q^nx_rz)\delta(q^{n+1}x_sw)\bigg(\theta_1(qx_r^{-1}x_s)(q^{-1}z-w)+(q^{-1}w-z)\bigg)=0.
\]
Since 
\[(q^{-1}z-w)\delta(q^nzx_r)\delta(q^{n+1}wx_r)=0,\]
the $s=r$ term in $(\hat{B}_n(z) \hat{B}_{n+1}(w)f)(x_{[\mathbf{v}]})$ will have no contribution to the left hand side of \eqref{equ:BBnn+1}.
Thus,
\begin{align*}
    &(q^{-1}z-w)(\hat{B}_n(z) \hat{B}_{n+1}(w)f)(x_{[\mathbf{v}]})+(q^{-1}w-z)(\hat{B}_{n+1}(w)\hat{B}_n(z) f)(x_{[\mathbf{v}]})\\
    =&(q^{-1}z-w)\sum_{r\in [\bfv]_n}\delta(q^nzx_r)\delta(q^{n+1}w/x_r)\cdot \Phi_{[\bfv]_n\setminus\{r\}}(qx_r)\Phi_{[\bfv]_{n+1}\cup\{r\}}(qx_r^{-1}) \cdot  f(x_{[\mathbf{v}]})\\
    &+(q^{-1}w-z)\sum_{s\in [\bfv]_{n+1}}\delta(q^nz/x_s)\delta(q^{n+1}wx_s)\cdot \Phi_{[\bfv]_n}(qx_s^{-1})\Phi_{[\bfv]_{n+1}\setminus\{s\}}(qx_s) \cdot  f(x_{[\mathbf{v}]})\\
    =&(z-q^{-1}w)f(x_{[\mathbf{v}]})\bigg(\sum_{r\in [\bfv]_n}\delta(q^nzx_r)\delta(q^{n+1}w/x_r)\cdot \Phi_{[\bfv]_n\setminus\{r\}}(qx_r)\Phi_{[\bfv]_{n+1}}(qx_r^{-1})\\
    &-\sum_{s\in [\bfv]_{n+1}}\delta(q^nz/x_s)\delta(q^{n+1}wx_s)\cdot \Phi_{[\bfv]_n}(qx_s^{-1})\Phi_{[\bfv]_{n+1}\setminus\{s\}}(qx_s)\bigg)\\
    =&\frac{\delta(q^{2n+1}zw)}{1-q^2}\bigg((w-qz)\hat{\K}_{n+1} \hat{\bTh}_n(z)f+(z-qw)\hat{\K}_n \hat{\bTh}_{n+1}(w)f\bigg)(x_{[\bfv]}),
\end{align*}
where the last equality follows from Lemma \ref{lem:Thetan}.

\subsubsection{Relation \eqref{equ:iserre}}
We prove the equivalent Serre relation in Remark \ref{rem:iserre}. Since $c_{i,\tau i}=-1$, $i=n$ or $n+1$. Let us check the case $i=n$, as the other case can be checked similarly.

By definition,
\begin{align*}
    &(\hat{B}_{n}(w_1)\hat{B}_{n}(w_2)\hat{B}_{n+1}(z)f)(x_{[\mathbf{v}]})\\
    =&\sum_{r\neq s\in [\bfv]_n}\sum_{t\in [\bfv]_{n+1}\cup\{r,r',s,s'\}}\delta(q^nw_1x_r)\delta(q^nw_2x_s)\delta(q^{n+1}zx_t)\\
    &\cdot \Phi_{[\bfv]_n\setminus\{r\}}(qx_r)\Phi_{[\bfv]_n\setminus\{r,s\}}(qx_s)\Phi_{[\bfv]_{n+1}\cup\{r,r',s,s'\}\setminus\{t\}}(qx_t)\cdot f(x_{\tau_t^+\tau_s^+\tau_r^+[\bfv]}),
\end{align*}
\begin{align*}
    &(\hat{B}_{n}(w_1)\hat{B}_{n+1}(z)\hat{B}_{n}(w_2)f)(x_{[\mathbf{v}]})\\
    =&\sum_{r\in [\bfv]_n}\sum_{t\in [\bfv]_{n+1}\cup\{r,r'\}}\sum_{s\in [\bfv]_n\cup\{t'\}\setminus\{r\}}\delta(q^nw_1x_r)\delta(q^nw_2x_s)\delta(q^{n+1}zx_t)\\
    &\cdot \Phi_{[\bfv]_n\setminus\{r\}}(qx_r)\Phi_{[\bfv]_n\cup\{t'\}\setminus\{r,s\}}(qx_s)\Phi_{[\bfv]_{n+1}\cup\{r,r'\}\setminus\{t\}}(qx_t)\cdot f(x_{\tau_s^+\tau_t^+\tau_r^+[\bfv]}),
\end{align*}
and 
\begin{align*}
    &(\hat{B}_{n+1}(z)\hat{B}_{n}(w_1)\hat{B}_{n}(w_2)f)(x_{[\mathbf{v}]})\\
    =&\sum_{t\in [\bfv]_{n+1}}\sum_{r\neq s\in [\bfv]_n\cup\{t'\}}\delta(q^nw_1x_r)\delta(q^nw_2x_s)\delta(q^{n+1}zx_t)\\
    &\cdot \Phi_{[\bfv]_n\cup\{t'\}\setminus\{r\}}(qx_r)\Phi_{[\bfv]_n\cup\{t'\}\setminus\{r,s\}}(qx_s)\Phi_{[\bfv]_{n+1}\setminus\{t\}}(qx_t)\cdot f(x_{\tau_s^+\tau_r^+\tau_t^+[\bfv]}).
\end{align*}

Therefore, for $r\neq s\in [\bfv]_n$ and $t\in [\bfv]_{n+1}$, the coefficient of 
\[\delta(q^nw_1x_r)\delta(q^nw_2x_s)\delta(q^{n+1}zx_t)\cdot \Phi_{[\bfv]_n\setminus\{r,s\}}(qx_r)\Phi_{[\bfv]_n\setminus\{r,s\}}(qx_s)\Phi_{[\bfv]_{n+1}\setminus\{t\}}(qx_t)\cdot f(x_{\tau_s^+\tau_r^+\tau_t^+[\bfv]})\]
in
\[\bigg(\mathbb{S}_{n,n+1}(w_1,w_2| z)f\bigg)(x_{[\mathbf{v}]})\]
is $0$ because of the following identity:
\begin{align*}
    &\theta_1(qx_r/x_s)\bigg(\theta_1(qx_t/x_r)\theta_1(qx_tx_r)\theta_1(qx_t/x_s)\theta_1(qx_tx_s)\\
    &-[2]\theta_1(qx_t/x_r)\theta_1(qx_tx_r)\theta_1(qx_tx_s)+\theta_1(qx_tx_r)\theta_1(qx_tx_s)\bigg)\\
    &+\theta_1(qx_s/x_r)\bigg(\theta_1(qx_t/x_r)\theta_1(qx_tx_r)\theta_1(qx_t/x_s)\theta_1(qx_tx_s)\\
    &-[2]\theta_1(qx_t/x_s)\theta_1(qx_tx_r)\theta_1(qx_tx_s)+\theta_1(qx_tx_r)\theta_1(qx_tx_s)\bigg)\\
    =&[2]\theta_1(qx_tx_r)\theta_1(qx_tx_s)\bigg(\theta_1(qx_t/x_r)\theta_1(qx_t/x_s)-\theta_1(qx_r/x_s)\theta_1(qx_t/x_r)\\
    &-\theta_1(qx_s/x_r)\theta_1(qx_t/x_s)+1\bigg)\\
    =&0.
\end{align*}
Here we have used the identity $\theta_1(qz)+\theta_1(qz^{-1})=[2]$.

Moreover, for $r\neq s\in [\bfv]_n$, the coefficient of 
\[\delta(q^nw_1x_r)\delta(q^nw_2x_s)\delta(q^{n+1}zx_r)\cdot \Phi_{[\bfv]_n\setminus\{r,s\}}(qx_r)\Phi_{[\bfv]_n\setminus\{r,s\}}(qx_s)\Phi_{[\bfv]_{n+1}}(qx_r)\cdot f(x_{\tau_s^+\tau_r^+\tau_r^+[\bfv]})\]
in
\[\bigg(\mathbb{S}_{n,n+1}(w_1,w_2| z)f\bigg)(x_{[\mathbf{v}]})\]
is $0$ because of the following identity
\begin{align*}
    &\theta_1(qx_r/x_s)\theta_1(qx_r^2)\theta_1(qx_r/x_s)\theta_1(qx_rx_s)+\theta_1(qx_s/x_r)\theta_1(qx_r^2)\theta_1(qx_r/x_s)\theta_1(qx_rx_s)\\  
    &-[2]\theta_1(qx_r/x_s)\theta_1(qx_sx_r)\theta_1(qx_r^2)=0.
\end{align*}

Therefore,
\begin{align*}
    &\bigg(\mathbb{S}_{n,n+1}(w_1,w_2| z)f\bigg)(x_{[\mathbf{v}]})\\
    =&\Sym_{w_1,w_2}\bigg(\sum_{r\neq s\in [\bfv]_n}\delta(q^nw_1x_r)\delta(q^nw_2x_s)\delta(q^{n+1}zx_r^{-1})\\
    &\cdot \Phi_{[\bfv]_n\setminus\{r\}}(qx_r)\Phi_{[\bfv]_n\setminus\{r,s\}}(qx_s)\Phi_{[\bfv]_{n+1}\cup\{r,s,s'\}}(qx_r^{-1})\cdot f(x_{\tau_s^+[\bfv]})\\  
    &+\sum_{r\neq s\in [\bfv]_n}\delta(q^nw_1x_r)\delta(q^nw_2x_s)\delta(q^{n+1}zx_s^{-1})\\
    &\cdot \Phi_{[\bfv]_n\setminus\{r\}}(qx_r)\Phi_{[\bfv]_n\setminus\{r,s\}}(qx_s)\Phi_{[\bfv]_{n+1}\cup\{r,r',s\}}(qx_s^{-1})\cdot f(x_{\tau_r^+[\bfv]})\\  
    &-[2]\sum_{r\in [\bfv]_n}\sum_{t\in [\bfv]_{n+1}}\delta(q^nw_1x_r)\delta(q^nw_2/x_t)\delta(q^{n+1}zx_t)\\
    &\cdot \Phi_{[\bfv]_n\setminus\{r\}}(qx_r)\Phi_{[\bfv]_n\setminus\{r\}}(qx_t^{-1})\Phi_{[\bfv]_{n+1}\cup\{r,r'\}\setminus\{t\}}(qx_t)\cdot f(x_{\tau_r^+[\bfv]})\\ 
    &-[2]\sum_{r\in [\bfv]_n}\delta(q^nw_1x_r)\delta(q^nw_2x_r^{-1})\delta(q^{n+1}zx_r)\\  
    &\cdot \Phi_{[\bfv]_n\setminus\{r\}}(qx_r)\Phi_{[\bfv]_n\setminus\{r\}}(qx_r^{-1})\Phi_{[\bfv]_{n+1}\cup\{r'\}}(qx_r)\cdot f(x_{\tau_r^+[\bfv]})\\
    &-[2]\sum_{r\in [\bfv]_n}\sum_{s\in [\bfv]_n}\delta(q^nw_1x_r)\delta(q^nw_2x_s)\delta(q^{n+1}zx_r^{-1})\\
    &\cdot \Phi_{[\bfv]_n\setminus\{r\}}(qx_r)\Phi_{[\bfv]_n\setminus\{s\}}(qx_s)\Phi_{[\bfv]_{n+1}\cup\{r\}}(qx_r^{-1})\cdot f(x_{\tau_s^+[\bfv]})\\  
    &+\sum_{t\in [\bfv]_{n+1}}\sum_{s\in [\bfv]_n}\delta(q^nw_1x_t^{-1})\delta(q^nw_2x_s)\delta(q^{n+1}zx_t)\\
    &\cdot \Phi_{[\bfv]_n}(qx_t^{-1})\Phi_{[\bfv]_n\setminus\{s\}}(qx_s)\Phi_{[\bfv]_{n+1}\setminus\{t\}}(qx_t)\cdot f(x_{\tau_s^+[\bfv]})\\  
    &+\sum_{t\in [\bfv]_{n+1}}\sum_{r\in [\bfv]_n}\delta(q^nw_1x_r)\delta(q^nw_2x_t^{-1})\delta(q^{n+1}zx_t)\\
    &\cdot \Phi_{[\bfv]_n\cup\{t'\}\setminus\{r\}}(qx_r)\Phi_{[\bfv]_n\setminus\{r\}}(qx_t^{-1})\Phi_{[\bfv]_{n+1}\setminus\{t\}}(qx_t)\cdot f(x_{\tau_r^+[\bfv]})\bigg)  \\
    =&[2]\Sym_{w_1,w_2}\sum_{r\in [\bfv]_n}\delta(q^nw_1x_r)\Phi_{[\bfv]_n\setminus\{r\}}(qx_r)\cdot f(x_{\tau_r^+[\bfv]})\\
    &\bigg(\sum_{s\in [\bfv]_n\setminus\{r\}}\delta(q^nw_2x_s)\delta(q^{n+1}zx_s^{-1})\cdot \Phi_{[\bfv]_n\setminus\{r,s\}}(qx_s)\Phi_{[\bfv]_{n+1}\cup\{r,s\}}(qx_s^{-1})\\  
    &-\sum_{t\in [\bfv]_{n+1}\cup\{r,r'\}}\delta(q^nw_2/x_t)\delta(q^{n+1}zx_t)\Phi_{[\bfv]_n\setminus\{r\}}(qx_t^{-1})\Phi_{[\bfv]_{n+1}\cup\{r,r'\}\setminus\{t\}}(qx_t)\\ 
    &-\sum_{s\in [\bfv]_n\setminus\{r\}}\delta(q^nw_2x_s)\delta(q^{n+1}zx_s^{-1})\Phi_{[\bfv]_n\setminus\{s\}}(qx_s)\Phi_{[\bfv]_{n+1}\cup\{s\}}(qx_s^{-1})\\  
    &+\sum_{t\in [\bfv]_{n+1}}\delta(q^nw_2x_t^{-1})\delta(q^{n+1}zx_t)\Phi_{[\bfv]_n\setminus\{r\}}(qx_t^{-1})\Phi_{[\bfv]_{n+1}\setminus\{t\}}(qx_t)\bigg).
\end{align*}
Using the identities $\delta(z/w)f(z)=\delta(z/w)f(w)$, and
\[\theta_1(zw_1^{-1})\delta(q^nw_1x_r)\delta(q^{n+1}zx_r)=\theta_1(q^{-1}w_1w_2^{-1})\delta(q^nw_1x_r)\delta(q^nw_2x_r)=0,\] the above equals
\begin{align*}
    &[2]\Sym_{w_1,w_2}\sum_{r\in [\bfv]_n}\delta(q^nw_1x_r)\Phi_{[\bfv]_n\setminus\{r\}}(qx_r)\cdot f(x_{\tau_r^+[\bfv]})\\
    &\bigg(\theta_1(w_2z^{-1})\theta_1(q^{-1}w_1w_2^{-1})\sum_{s\in [\bfv]_n\setminus\{r\}}\delta(q^nw_2x_s)\delta(q^{n+1}zx_s^{-1})\cdot \Phi_{[\bfv]_n\setminus\{r,s\}}(qx_s)\Phi_{[\bfv]_{n+1}\cup \{r,r'\}}(qx_s^{-1})\\ 
    &-\sum_{t\in [\bfv]_{n+1}\cup\{r,r'\}}\delta(q^nw_2x_t^{-1})\delta(q^{n+1}zx_t)\Phi_{[\bfv]_n\setminus\{r\}}(qx_t^{-1})\Phi_{[\bfv]_{n+1}\cup\{r,r'\}\setminus\{t\}}(qx_t)\\ 
    &-\theta_1(qw_1w_2^{-1})\theta_1(w_2z^{-1})\theta_1(zw_1^{-1})\theta_1(q^{-1}w_1w_2^{-1})\\
    &\cdot \sum_{s\in [\bfv]_n\setminus\{r\}}\delta(q^nw_2x_s)\delta(q^{n+1}zx_s^{-1})\Phi_{[\bfv]_n\setminus\{r,s\}}(qx_s)\Phi_{[\bfv]_{n+1}\cup\{r,r'\}}(qx_s^{-1})\\ 
    &+\theta_1(zw_1^{-1})\theta_1(q^{-1}w_1w_2^{-1})\\
    &\cdot \sum_{t\in [\bfv]_{n+1}\cup\{r,r'\}}\delta(q^nw_2x_t^{-1})\delta(q^{n+1}zx_t)\Phi_{[\bfv]_n\setminus\{r\}}(qx_t^{-1})\Phi_{[\bfv]_{n+1}\cup\{r,r'\}\setminus\{t\}}(qx_t)\bigg) \\
    =&[2]\Sym_{w_1,w_2}\frac{(1-q^2)w_1(z-qw_2)}{(z-qw_1)(w_1-q^2w_2)}\sum_{r\in [\bfv]_n}\delta(q^nw_1x_r)\Phi_{[\bfv]_n\setminus\{r\}}(qx_r)\cdot f(x_{\tau_r^+[\bfv]})\\
    &\bigg(\sum_{s\in [\bfv]_n\setminus\{r\}}\delta(q^nw_2x_s)\delta(q^{n+1}zx_s^{-1})\cdot \Phi_{[\bfv]_n\setminus\{r,s\}}(qx_s)\Phi_{[\bfv]_{n+1}\cup \{r,r'\}}(qx_s^{-1})\\ 
    &-\sum_{t\in [\bfv]_{n+1}\cup\{r,r'\}}\delta(q^nw_2x_t^{-1})\delta(q^{n+1}zx_t)\Phi_{[\bfv]_n\setminus\{r\}}(qx_t^{-1})\Phi_{[\bfv]_{n+1}\cup\{r,r'\}\setminus\{t\}}(qx_t)\bigg) \\
    =&[2]\Sym_{w_1,w_2}\frac{qw_1\delta(q^{2n+1}zw_2)}{(z-qw_1)(q^2w_2-w_1)}\\
    &\Bigg(\hat{B}_n(w_1)\bigg((z-qw_2)\hat{\K}_{n+1} \hat{\bTh}_n(w_2)+(w_2-qz)\hat{\K}_n \hat{\bTh}_{n+1}(z)\bigg)f\Bigg)(x_{[\bfv]}).
\end{align*}
Here the first equality follows from
\begin{align*}
    &1-\theta_1(zw_1^{-1})\theta_1(q^{-1}w_1w_2^{-1})\\
    =&\theta_1(w_2z^{-1})\theta_1(q^{-1}w_1w_2^{-1})-\theta_1(qw_1w_2^{-1})\theta_1(w_2z^{-1})\theta_1(zw_1^{-1})\theta_1(q^{-1}w_1w_2^{-1})\\
    =&\frac{(1-q^2)w_1(z-qw_2)}{(z-qw_1)(w_1-q^2w_2)},
\end{align*}
and the last equality follows from Lemma \ref{lem:Thetan}.
This finishes the proof of Relation \eqref{equ:iserre}, and hence concludes the proof of Theorem \ref{thm:polyrep}.

\appendix
\section{Construction via type D Steinberg varieties} \label{append}

\subsection{Orthogonal group}
Let $V=\C^{2d}$ be endowed with a non-degenerate symmetric bilinear form $(-,-)$. Let $G=\mathrm{O}(V)$ be the orthogonal group of $V$ under this form, realized as a matrix group associated with a fixed basis $\{e_1, e_2,\ldots,e_{2d}\}$ of $V$ such that $(e_i,e_j)=\delta_{i,2d+1-j}$. 
Its identity connected component $G^\circ=\mathrm{SO}(V)$ is the special orthogonal group.

We take the maximal torus $T$ (resp. Borel subgroup $B$) of $G$ as the subgroup consisting of all diagonal matrices (resp. uppertriangle matrices) lying in $G$. It is a standard fact that $T,B \subset G^\circ$.  
 
The Weyl group $W_\mathfrak{c}=N_G(T)/T$ for $G$ is of type $C_d$ (caution: NOT of type $D_d$), which can be regarded as the subgroup of $G$ generated by
\begin{align*}
            s_i = {\tiny \begin{pmatrix}
                I_{i-1}\\
                & 0 & 1\\
                & 1 & 0\\
                &&& I_{2d-2i-2}\\
                &&&& 0 & 1\\
                &&&& 1 & 0\\
                &&&&&& I_{i-1}
            \end{pmatrix}} \ (i<d), \quad
            s_{d} = {\tiny \begin{pmatrix}
                I_{d-1}\\
                & 0 & 1\\
                & 1 & 0\\
                &&& I_{d-1}
            \end{pmatrix}}.
        \end{align*} 
Meanwhile, the Weyl group $W_\mathfrak{d}=N_{G^\circ}(T)/T$ for $G^\circ$ is of type $D_d$, which is a subgroup of $W_\mathfrak{c}$ generated by $s_i$ $(i<d)$ and 
$\tilde{s}_d=s_ds_{d-1}s_d$.

\subsection{Partial flag variety of type D}

For our purpose (i.e., to realize $\widetilde{\mathbf{U}}^\imath_{2n}$), we study the {\em $2n$-step isotropic flag variety}:
$$
\mathscr{F}=\{ F=( 0=V_0 \subset V_1 \subset \cdots \subset V_{2n}=V)\ \mid \  V_i=V_{2n-i}^{\perp},\  \forall i \},
$$
which can be decomposed into a disjoint union of $G$-orbits:
$\mathscr{F} = \bigsqcup_{\mathbf{v} \in \Lambda_{2n,d}^\imath} \mathscr{F}_{\mathbf{v}}$, where 
 \begin{align*}
\Lambda_{2n,d}^\imath= \big\{ \mathbf{v}=(v_i) \in \mathbb{N}^{2n} \mid   v_i = v_{2n+1-i},\quad \textstyle \sum_{i=1}^{2n} v_i = 2d \big\}
\end{align*}
and each $G$-orbit
\begin{align*}
\mathscr{F}_{\mathbf{v}}=\{ F=( 0=V_0 \subset V_1 \subset \cdots \subset V_{2n}=V)\ \mid \  V_i=V_{2n-i}^{\perp} ,\ \text{dim}(V_{i}/V_{i-1})= v_i,\  \forall i \}. 
\end{align*}

Most of the notations and concepts introduced in Section~\ref{sec:Stein} (e.g. $[\mathbf{v}]_i$,  $F_\mathbf{v}$, $P_\mathbf{v}$, etc.) carry over to the current setting. We will not reinterpret them here and shall use them directly. 

In \cite{DLZ25}, Du, Li and Zhao studied the $(2n+1)$-step isotropic flag variety. Each $2n$-step isotropic flag can be regarded as a $(2n+1)$-step one by inserting a copy of the $n$-th subspace at the $(n+1)$-th step, since that subspace is always Lagrangian. Therefore, some arguments/results for $(2n+1)$-step isotropic flags apply to our setting. 

We have $$\mathscr{F}_{\mathbf{v}}\simeq G/P_{\mathbf{v}}$$ as well, identifying $gP_{\mathbf{v}}\in G/P_{\mathbf{v}}$ with $gF_{\mathbf{v}}\in\mathscr{F}_{\mathbf{v}}$.
Furthermore, it follows from \cite[Lemma~3.2]{DLZ25} that $P_\mathbf{v}\subset G^\circ$, by regarding $\mathbf{v}=(v_i)\in\Lambda_{2n,d}^\imath$ as the element $\lambda=(v_1,v_2,\ldots,v_n,0)$ therein.

Similarly to \cite[Corollary~3.7]{DLZ25},
the map $(F,F')\mapsto A=(a_{ij})_{2n\times 2n}$ with $a_{ij}=\dim\frac{V_i\cap V'_j}{V_{i-1}\cap V'_j+V_i\cap V'_{j-1}}$ induces a bijection between the set of $G$-orbits on $\mathscr{F} \times \mathscr{F}$ and the set 
\begin{align*}
\Xi_{2n,d}=\Big\{A=(a_{ij}) \in {\rm Mat}_{2n\times 2n}(\mathbb{N})\ | \ \sum_{i,j}a_{ij}=2d,\ a_{ij}=a_{2n+1-i,2n+1-j},\ \forall \  i, j \Big\}.
\end{align*}
Furthermore, the diagonal $G$-orbits on $\mathscr{F}_{\mathbf{v}}\times \mathscr{F}_{\mathbf{w}}$ are in bijection with the double cosets $W_{[\mathbf{v}]}\backslash W_\mathfrak{c}/W_{[\mathbf{w}]}$, which corresponds 
$A\in\Xi_{2n,d}$ to a triple $(\mathbf{v}, w_A, \mathbf{w})$ such that 
 $\mathbf{v}=\mathrm{ro}(A)$, $\mathbf{w}=\mathrm{co}(A)$ and $w_A$ is a minimal double coset representative of $W_{[\mathbf{v}]}\backslash W_\mathfrak{c}/W_{[\mathbf{w}]}$ determined uniquely by $A$ (cf. \cite[\S2.2]{LL21}).

We denote by $\mathcal{O}_A$ the diagonal $G$-orbit on $\mathscr{F}\times \mathscr{F}$ associated with $A\in \Xi_{2n,d}$. Each $G$-orbit $\mathcal{O}_A$ consists of two $G^\circ$-orbits: one is generated by $(F_\mathbf{v},w_AF_\mathbf{w})$ and the other is generated by $(s_d F_\mathbf{v},s_d w_AF_\mathbf{w})$, where $\mathbf{v}=\mathrm{ro}(A)$ and $\mathbf{w}=\mathrm{co}(A)$. We denote them by $\mathcal{O}_A^{\mathbb{1}}$ and 
$\mathcal{O}_A^{s_d}$, respectively.
They lie in different connected components of $\mathscr{F}_\mathbf{v} \times \mathscr{F}_\mathbf{w}$.

Define a order $\preceq$ on $\Xi_{2n,d}$ as follows.
 For $A, B \in \Xi_{2n,d}$, we say $A\preceq B$ if and only if the condition \eqref{equ:order} and the following hold: 
$$\sum_{1\leq i ,j\leq n}a_{ij} \equiv \sum_{1\leq i,j \leq n}b_{ij} \mod 2.$$
\begin{prop}\label{prop:Dorder}
For any $A,B\in\Xi_{2n,d}$, we have $A\preceq B$ 
if and only if $\mathcal{O}_A\subseteq\overline{\mathcal{O}}_B$.
\end{prop}
\begin{proof}
    It follows from \cite[Lemma~4.6.1]{FL15} that $A\preceq B$ if and only if $\mathcal{O}_A^{\mathbb{1}}\subseteq\overline{\mathcal{O}}_B^{\mathbb{1}}$ if and only if $\mathcal{O}_A^{s_d}\subseteq\overline{\mathcal{O}}_B^{s_d}$. Therefore, $A\preceq B$ if and only if $\mathcal{O}_A\subseteq\overline{\mathcal{O}}_B$.
\end{proof}

For $1\leq i<2n$ and $\mathbf{v}''\in\Lambda_{2n,d-1}^\imath$, denote $$\mathbf{v}=\mathbf{v}''+\mathbf{e}_i+\mathbf{e}_{\tau i+1},\quad \mathbf{v}'=\mathbf{v}''+\mathbf{e}_{i+1}+\mathbf{e}_{\tau i}\in\Lambda_{2n,d}^\imath \quad\text{and}$$ 
$$E_{i,i+1}^\theta(\mathbf{v}'',1)=\mathrm{diag}(\mathbf{v}'')+E_{i,i+1}+E_{2n+1-i,2n-i}\in\Xi_{2n,d},$$
where $\mathbf{e}_i\in\mathbb{N}^{2n}$ is the vector with $1$ at $i$-th position and $0$ otherwise, $E_{ij}$ is the standard $2n\times 2n$ matrix with $1$ at $(i,j)$-th entry and $0$ otherwise.
The $G$-orbit associated with $E_{i,i+1}^\theta(\mathbf{v}'',1)$ ($1\leq i< 2n$) is given by 
\begin{align*}
\mathcal{O}_{E_{i,i+1}^{\theta}(\mathbf{v}'',1)}&=\bigg\{(F,F')\mid 
\substack{F=(V_k)_{0\le k \le 2n} 
\in \mathscr{F}_{\mathbf{v} }\\ F'=(V'_k)_{0\le k \le 2n}
\in \mathscr{F}_{\mathbf{v}' }}, V'_i\stackrel{1}{\subset} V_i, V_k=V'_k \text{ if } k\neq i \text{ or } 2n-i\bigg\}\quad (i\neq n),\\
\mathcal{O}_{E_{n,n+1}^{\theta}(\mathbf{v}'',1)}&=\bigg\{(F,F')\mid 
\substack{F=(V_k)_{0\le k \le 2n} 
\in \mathscr{F}_{\mathbf{v} }\\ F'=(V'_k)_{0\le k \le 2n}
\in \mathscr{F}_{\mathbf{v}} }, V'_n\cap V_n\stackrel{1}{\subset} V_n, V_k=V'_k \text{ if } k\neq n\bigg\}.
\end{align*}

\begin{lemma}\label{lem:proper}
The $G$-orbits $\mathcal{O}_{E_{i,i+1}^{\theta}(\mathbf{v}'',1)}$ ($1\leq i< 2n$) in $\mathscr{F}\times\mathscr{F}$ are all closed.
\end{lemma}
\begin{proof}
The closedness of $\mathcal{O}_{E_{i,i+1}^{\theta}(\mathbf{v}'',1)}$ follows from Proposition~\ref{prop:Dorder} and the minimality of $E_{i,i+1}^{\theta}(\mathbf{v}'',1)$ under the order $\preceq$.
\end{proof}



\subsection{Construction of operators}
Let
$$Z:=T^*\mathscr{F}\times_{\mathcal{N}}T^*\mathscr{F}$$
be the {\em Steinberg variety} associated to $\mathscr{F}$, where $\mathcal{N}$ is the nilpotent cone of $\mathfrak{so}_{2n}=\mathrm{Lie}(G)$.
For $A\in\Xi_{2n,d}$, let $Z_A=T_{\mathcal{O}_A}^*(\mathscr{F}_\mathbf{v}\times\mathscr{F}_\mathbf{w})$ denote the conormal bundle of $\mathcal{O}_A$, where $\mathbf{v}=\mathrm{ro}(A)$ and $\mathbf{w}=\mathrm{co}(A)$.

Thanks to Lemma~\ref{lem:proper}, we can introduce geometric $B$-operators as the same as in \S\ref{sec:GeoB}. That is, for $1\leq i\leq n$, $\mathbf{v}\in\Lambda_{2n,d}^\imath$ and $r\in\mathbb{Z}$, define
\[\mathscr{B}_{i,\mathbf{v},r} =  \pi^{*}(\mathrm{Det}(T^{*}_{p_1})\otimes \mathcal{L}_{\mathbf{v},i}^{\otimes r})\in K^{G\times \bbC^*}(Z_{E_{i,i+1}^{\theta}(\mathbf{v}-\mathbf{e}_i-\mathbf{e}_{2n+1-i},1)})\]
and 
\[\mathscr{B}_{i,r}= \sum_{\mathbf{v}} (-q)^{1-v_i}\mathscr{B}_{i,\mathbf{v},r}\in K^{G\times \bbC^*}(Z).\]
where $\mathcal{L}_{\mathbf{v},i}$ is the tautological line bundle on $\mathcal{O}_{E_{i,i+1}^{\theta}(\mathbf{v}-\mathbf{e}_{i}-\mathbf{e}_{2n+1-i},1)}$ whose stalk at a point $(F,F')$ is given by $V_i/V'_i$ if $i\neq n$, and by $V_n/V_n\cap V_n'$ if $i=n$.

Let us modify some notations used in the main text to adopt $2n$ instead of $2n+1$, so that they are consistent with the definitions in \cite{SW24}, for example: 
$$[\mathbf{v}]:=([\mathbf{v}]_1,[\mathbf{v}]_2,\ldots,[\mathbf{v}]_{2n}) \quad \text{for } \mathbf{v}\in\Lambda_{2n,d}^\imath $$
and
\begin{align*}
[A]:=([A]_{11},\ldots,[A]_{1,2n},[A]_{21},\ldots,[A]_{2n,2n}) \quad\text{for } A\in\Xi_{2n,d}.
 \end{align*}

Although the geometric structure for the type D setting here is different from the one for the type C setting presented in \cite{SW24}, their combinatorial structures are almost the same (under $N=2n$).
In order to shorten the appendix, we directly use the results on combinatorial structures shown in \cite{SW24} without providing proofs, such as
\begin{align*}
\mathbf{R}&:=\mathbb{C}[x_1^{\pm1},x_2^{\pm1},\ldots,x_d^{\pm1}]\simeq K^G(\mathscr{B}),\quad \mathbf{R}^{[\mathbf{v}]}\simeq K^G(\mathscr{F}_{\mathbf{v}}),\quad \mathbf{R}^{[A]}\simeq K^G(\mathcal{O}_A),\\
\mathbf{P}&:=\bigoplus_{\mathbf{v}\in\Lambda_{2n,d}^\imath}\mathbf{R}^{[\mathbf{v}]}[q,q^{-1}]\simeq K^{G\times\mathbb{C}^*}(\mathscr{F})\simeq K^{G\times\mathbb{C}^*}(T^*\mathscr{F}).
\end{align*}

Define the operators $\hat{\mathbb{K}}_i$ ($1\leq i<2n$), acting on $\mathbf{R}^{[\mathbf{v}]}[q,q^{-1}]$ by the scalar $(-q)^{\delta_{ni}}q^{v_{i+1}-v_i}$, which is equivalent to \cite[(4.39)]{SW24}.

Define the operators $\hat{\Theta}_{i,r}$ ($1\leq i<2n, r \geq 0$) on $\mathbf{P}$ given by $$(\hat{\Theta}_i(z)f)(x_{\mathbf{v}})=q^{v_{i+1}-v_{i}+\delta_{ni}}\theta_1(q^{2n-1}z^{2})^{\delta_{ni}}\Phi_{[\mathbf{v}]_{i}}(q^{1-i}z^{-1})\cdot \Phi_{[\mathbf{v}]_{i+1}}(q^{-1-i}z^{-1})^{-1}\cdot f(x_{\mathbf{v}}),$$
where $\hat{\Theta}_i(z)=1+(q-q^{-1})\sum_{r \geq 1} \hat{\Theta}_{i,r}z^r$.

Define the operators $\hat{B}_{i,r}$ ($1\leq i<2n, r\in\mathbb{Z}$), acting on $\mathbf{P}$ by
$$(\hat{B}_{i,r}f)(x_{[\mathbf{v}]})=\sum_{j \in [\mathbf{v}]_i} x_j^r\Phi_{[v]_i\backslash\{j\}}(qx_j)\cdot f(x_{\tau_j^+[\mathbf{v}]}).$$
Denote $\hat{B}_i(z):=\sum_{r\in\mathbb{Z}}q^{ri} \hat{B}_{i,r}z^r$.
\begin{prop}\label{prop:AppBir}
The convolution action of $\mathscr{B}_{i,r}\in K^{G\times\mathbb{C}^*}(Z)$ on $\mathbf{P}\simeq K^{G\times\mathbb{C}^*}(T^*\mathscr{F})$ is given by $\hat{B}_{i,r}$.
\end{prop}
\begin{proof}
The case of $i\neq n$ is proved in \cite[Proposition 6.1 and 6.3]{SW24}. For the case of $i=n$, we still let $p_1$ denote the first projection from the orbit $\mathcal{O}_{E_{n,n+1}^{\theta}(\mathbf{v}'',1)}$ to $\mathscr{F}_{\mathbf{v}}$. Then the fiber is
\[p_1^{-1}(F)\simeq\{V_n'\subset V\mid V_n'=(V_n')^\perp, \dim V_n/(V_n\cap V_n')=1\}.\]
Since we are in the orthogonal case, the above fiber is isomorphic to $\Gr(v_n-1,V_n)$, i.e., the Lagrangian subspace $V_n'$ is uniquely determined by $V_n\cap V_n'$. Then the same argument as in the proof of Proposition~\ref{prop:Bir} gives the formula for $\mathscr{B}_{n,r}$.
\end{proof}

\begin{remark}\label{rem:comp}
    Let us compare our operators with the ones from \cite{SW24}. First of all, the $\hat{\mathbb{K}}_i$'s are the same. For $i\neq n$,  the $\hat{\Theta}_{i,r}$ and $\hat{B}_{i,r}$ are the same. For $i=n$, the $\hat{B}_n(z)$ in \textit{loc. cit.} coincides with our $\theta_1(q^{1-2n}z^{-2})\hat{B}_n(z)$ (see \cite[(4.38)]{SW24}), while the $\hat{\Theta}_n(z)$ \footnote{Here we add back the factor $\frac{1-qCz^2}{1-Cz^2}=q\cdot\theta_1(qCz)$ in Lemma 4.3 in \textit{loc. cit.} to the operator $\hat{\Theta}_n(z)$.} in \textit{loc. cit.} coincides with our $\theta_1(q^{1+2n}z^2)\theta_1(q^{1-2n}z^{-2})\hat{\Theta}_n(z)$.
\end{remark}

\subsection{Affine iquantum groups}
Let $(c_{ij})_{0\leq i,j<2n}$ be the Cartan matrix of affine type $\tilde{A}_{2n-1}$. Let $\tau$ be the diagram involution such that $\tau(0)=0$ and $\tau(i)=2n-i$ for $1\leq i<2n$. These data give the quasi-split Satake diagram: 
\begin{figure}[htbp]
\centering
\begin{tikzpicture}[scale=.4]

\draw[blue, thick, ->] (-4.5, -1) arc[start angle=30, end angle=330, radius=0.7];

\draw[thick] (-4, -1.5) circle (0.3);
\node at (-4, -2.5) {$0$};

\node at (0,0.75) {$1$};
\node at (4,0.75) {$2$};
\node at (12,0.75) {$n-1$};    

\node at (0,-3.75) {$2n-1$};
\node at (4,-3.75) {$2n-2$};
\node at (12,-3.75) {$n+1$};
    
\node at (8,0) {$\dots$};
\node at (8,-3) {$\dots$};
    
\foreach \x in {0,2,6} 
{\draw[thick,xshift=\x cm] (\x, 0) circle (0.3); 
\draw[thick,xshift=\x cm] (\x, -3) circle (0.3);}

\draw[thick] (12,0) circle (0.3);
\draw[thick] (12,-3) circle (0.3);

\draw[thick] (-0.28, -0.11) -- (-3.7, -1.22);   
\draw[thick] (-0.28, -2.89) -- (-3.7, -1.78);  

\foreach \x in {0}
{\draw[thick,xshift=\x cm] (\x,0) ++(0.5,0) -- +(3,0);
\draw[thick,xshift=\x cm] (\x,-3) ++(0.5,0) -- +(3,0);}
   
\foreach \x in {2,4.5}
{\draw[thick,xshift=\x cm] (\x,0) ++(0.5,0) -- +(2,0);
\draw[thick,xshift=\x cm] (\x,-3) ++(0.5,0) -- +(2,0);}
    
\draw[thick] (12.28, -0.11) -- (15.72, -1.22);
\draw[thick] (12.28, -2.89) -- (15.72, -1.78);
    
\draw[thick] (16,-1.5) circle (0.3);
\node at (16,-2.5) {$n$};

\draw[blue, thick, <-] (16.5, -1.8) arc[start angle=210, end angle=510, radius=0.7];
    
\foreach \x in {0,4} 
\draw[thick,<->, blue, bend left=50] (\x+0.3,-2.5) to (\x+0.3,-0.5);
\draw[thick,<->, blue, bend left=50] (11.7,-2.5) to (11.7,-0.5);

\end{tikzpicture}.
\caption{Affine type $\mathrm{AIII}_{2n-1}^{(\tau)}$}
\label{fig:AIII2n-1}
\end{figure}
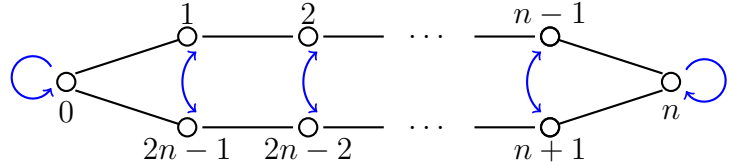

Let 
$\widetilde{\mathbf{U}}^\imath=\widetilde{\mathbf{U}}^\imath_{2n}$ be the affine iquantum group corresponding to this Satake diagram. A Drinfeld new presentation of $\widetilde{\mathbf{U}}^\imath$ was given in \cite{LWZ24}, saying that $\widetilde{\mathbf{U}}^\imath$ is isomorphic to the $\mathbb{C}(v)$-algebra generated by the elements $B_{il}$, $\Theta_{im}$, $\K_i^{\pm 1}$, and $C^{\pm 1}$, where $1\leq i< 2n$, $l\in\Z$ and $m>0$, subject to the relations \eqref{rel:1}-\eqref{rel:6}, \eqref{rel:8}, \eqref{rel:9}, and the following two relations:
\begin{align*}
&(v^2z-w) \bB_n(z) \bB_{n} (w) +(v^{2}w-z) \bB_{n}(w) \bB_n(z)\\
&= \frac{\bDel(zw)\K_n}{v-v^{-1}} 
 \big((z -v^{-2}w) \bTh_n (w) + (w -v^{-2}z) \bTh_n(z) \big),
\end{align*}
and
\begin{align*}
	\mathbb{S}_{n,j}(w_1,w_2| z)=&-\K_n\frac{\bDel(w_1w_2)}{v-v^{-1}}\mathrm{Sym}_{w_1,w_2}\big(\frac{[2]zw_1^{-1}}{1-v^2w_2w_1^{-1}}[\mathbf{\Theta}_n(w_2),\mathbf{B}_j(z)]_{v^{-2}}
    \\&+\frac{1+w_2w_1^{-1}}{1-v^2w_2w_1^{-1}}[\mathbf{B}_j(z),\mathbf{\Theta}_n(w_2)]_{v^{-2}}\big),\quad \text{if $j=n\pm1$}. 
\end{align*}
Here the generating functions $\mathbf{B}_i(z), \mathbf{\Theta}_i(z)$ and $\mathbf{\Delta}(z)$ are defined in $\eqref{def:genfunc}$ and $\mathbb{S}_{n,j}(w_1,w_2| z)$ is defined in \eqref{def:sym}.

\begin{theorem}
    The assignment
    \begin{align*}
        &v\mapsto q,\quad  C \mapsto q^{2n}, \quad \mathbb{K}_i \mapsto \hat{\mathbb{K}}_i, \quad 
        \mathbf{\Theta}_i(z) \mapsto \hat{\Theta}_i(z),\quad \mathbf{B}_i(z) \mapsto \hat{B}_i(z)
    \end{align*}
    defines a representation of $\widetilde{\mathbf{U}}^\imath_{2n}$ on $\mathbf{P}\otimes_{\mathbb{C}[q,q^{-1}]} \mathbb{C}(q)$. As a consequence, the assignment 
    $$v\mapsto q,\quad C\mapsto q^{2n}, \quad \mathbb{K}_i\mapsto \hat{\mathbb{K}}_i,\quad \Theta_{im}\mapsto \hat{\Theta}_{i,m},\quad B_{il}\mapsto \hat{B}_{i,l}$$
    extends to a $\mathbb{C}(q)$-algebra homomorphism
$$\widetilde{\mathbf{U}}^\imath_{2n}\to K^{G\times\mathbb{C}^*}(Z)\otimes_{\mathbb{C}[q,q^{-1}]}\mathbb{C}(q).$$
\end{theorem}
\begin{proof}
The first claim follows from \cite[Theorem 4.8]{SW24} and Remark \ref{rem:comp}. The second one follows from Proposition \ref{prop:AppBir}.
\end{proof}

\begin{remark}
Although the universal iquantum group $\widetilde{\mathbf{U}}^\imath_{2n}$ taken in this appendix is the same as that in \cite{SW24}, the equivariant K-groups $K^{G\times\mathbb{C}^*}(Z)$ in the two contexts are actually not isomorphic. Precisely, they are two specializations of the same three-parameter affine quantum Schur algebra of type C with different choices of parameters. This also explains, from another perspective, why we can establish the homomorphism from  $\widetilde{\mathbf{U}}^\imath_{2n}$ to $K^{G\times\mathbb{C}^*}(Z)\otimes_{\mathbb{C}[q,q^{-1}]}\mathbb{C}(q)$ without localization, unlike in \cite{SW24}. 
\end{remark}

\bibliographystyle{alpha}
\bibliography{k.bib}

\end{document}